\documentclass[arxiv]{agn_article}
\usepackage{csquotes}

\usepackage{amssymb,amsmath,amsthm,agn_bib}

\usepackage{agn_macros,agn_authors}
\usepackage{mathrsfs}
\usepackage{graphicx}
\usepackage{multicol}
\usepackage[inline]{enumitem}
\usepackage{array,multirow}
\usepackage[normalem]{ulem}
\usepackage{arydshln}
\usepackage[super]{nth}
\usepackage[d]{esvect}

\usepackage{nicematrix_neu,centernot,euscript}

\usepackage{sidecap}
\usepackage{caption}
\usepackage{subcaption}

\newcounter{alphasect}
\def\alphainsection{0}

\let\oldsection=\section
\def\section{%
  \ifnum\alphainsection=1%
    \addtocounter{alphasect}{1}
  \fi%
\oldsection}%
\renewcommand\thesection{%
  \ifnum\alphainsection=1%
    \Alph{alphasect}%
  \else%
    \arabic{section}%
  \fi%
}%

\newenvironment{alphasection}{%
  \ifnum\alphainsection=1%
    \errhelp={Let other blocks end at the beginning of the next block.}
    \errmessage{Nested Alpha section not allowed}
  \fi%
  \setcounter{alphasect}{0}
  \def\alphainsection{1}
}{%
  \setcounter{alphasect}{0}
  \def\alphainsection{0}
}%

\AtBeginBibliography{\footnotesize}

\newtheorem{theorem}{Theorem}[section]
\newtheorem{proposition}[theorem]{Proposition}
\newtheorem{lemma}[theorem]{Lemma}
\newtheorem{problem}[theorem]{Problem}
\newtheorem{corollary}[theorem]{Corollary}

\theoremstyle{definition}

\newtheorem{remark}[theorem]{Remark}
\newtheorem{example}[theorem]{Example}
\newtheorem*{akn*}{Acknowledgment}

\usepackage{xcolor,hyperref}
\definecolor{darkblue}{rgb}{0,0,0.6}
\definecolor{fg}{RGB}{34,139,34}  

\hypersetup{
    colorlinks=true,       
    linkcolor=darkblue,          
    citecolor=darkblue,        
    filecolor=darkblue,      
    urlcolor=darkblue           
}

\usepackage[colorinlistoftodos,prependcaption,textsize=tiny]{todonotes}
\newcommand{\skalarProd}[2]{\langle#1,#2\rangle}

\renewcommand{\D}{\operatorname{D}\hspace{-1pt}}

\newcommand{\Anti}{\operatorname{Anti}}

\newlist{thmenum}{enumerate}{1}
\setlist[thmenum]{label=\upshape(\alph*)}

\usepackage{tikz}

\begin{document}
\begin{tikzpicture}[remember picture, overlay]
 \node [xshift=-1cm,yshift=15cm,rotate=-90] at (current page.south east)
 {Proceedings of the Royal Society A, Mathematical, Physical and Engineering Sciences (2021), doi: \href{https://doi.org/10.1098/rspa.2021.0158}{10.1098/rspa.2021.0158}.
 };
\end{tikzpicture}
\numberwithin{equation}{section}

\numberwithin{equation}{section}

\title{On in-plane drill rotations for Cosserat surfaces}
\knownauthors[saem]{saem,lewintan,neff}

\maketitle
\begin{abstract}
\noindent
We show under some natural smoothness assumptions that pure in-plane drill rotations as deformation mappings of a $C^2$-smooth regular shell surface to another one parametrized over the same domain are impossible provided that the rotations are fixed at a portion of the boundary. Put otherwise, if the tangent vectors of the new surface are obtained locally by only rotating the given tangent vectors, and if these rotations have a rotation axis which coincides everywhere with the normal of the initial surface, then the two surfaces are equal provided they coincide at a portion of the boundary. In the language of differential geometry of surfaces we show that any isometry which leaves normals invariant and which coincides with the given surface at a portion of the boundary, is the identity mapping. 
\end{abstract}
\medskip


\section{Introduction}
In this contribution we consider the apparently novel question whether nontrivial pure in-plane drill rotations may appear in the deformation of shells if boundary conditions are prescribed that fix the Cosserat drill rotations at a portion of the boundary. 
\subsection{Main result}
\begin{proposition}\label{prop1}
	Let $\omega\subset \R^2$ be a bounded Lipschitz domain. Assume that $m,y_0\in C^2(\overline{\omega},\R^3)$ are regular surfaces, $Q\in C^1(\overline{\omega},\SO(3))$ and 
	\begin{align}
\nonumber	\D m(x)&=Q(x)\D y_0(x),\quad Q(x)\,n_0(x)=n_0(x)\,,\quad x\in \overline{\omega}\,,\\
 	 m|_{\gamma_{d}}&=y_0|_{\gamma_{d}}\,,
	\end{align}
	where $n_0=\frac{\partial_1 y_0\times \partial_2 y_0}{\norm{\partial_1 y_0\times \partial_2 y_0}}$ denotes the normal field on $y_0(\omega)$ and $\gamma_d$ is a relatively open, non-empty subset of the boundary $\partial\omega$.
	Then $m\equiv y_0$.
\end{proposition}
The interest for this question is not coming from differential geometry per se, but is motivated from shell models with independent director fields, so called Cosserat-surfaces \cite{cosserat1909theorie}. The additional field is a rotation vector field $Q\in\SO(3)$, necessitating additional balance equations and offering the possibility to introduce new (material) parameters into the model, coupling in-plane tangent vector fields and the rotation field $Q$ by a stiffness $\mu_c>0$. The question of how to determine the Cosserat couple modulus $\mu_c$ is  largely open in the dedicated literature \cite{agn_birsan2013characterization,agn_birsan2012equations,agn_birsan2013existence,agn_birsan2014existence,agn_birsan2016dislocation,chroscielewski2011modeling,ibrahimbegovic1994stress,pietraszkiewicz2011refined,pietraszkiewicz2016resultant,wisniewski2000kinematics,wisniewski1998shell,acharya2020rotations, agn_birsan2014shells, fox1992drill,hughes1989drilling,murin2013consistent,wisniewski2010finite,yeh1993shell}. We focus therefore on the effect, this Cosserat couple modulus $\mu_c$ may have and arrive at investigating \emph{pure in-plane drill rotations} for arbitrary shell surfaces. In the course of this investigation (see section \ref{moti} below) we connect the initial question to more standard rigidity results \cite{agn_lankeit2017integrability,agn_munch2008curl,acharya2000nonlinear,Ciarlet1988} for solid bodies and thin shells.

With the result of Proposition \ref{prop1}, in the end we are seeing that the stiffness $\mu_c$ in Cosserat shell models is arguably connected to a boundary condition and therefore, its status as material parameter is in doubt.

Proposition \ref{prop1} can be seen in the language of classical differential geometry of surfaces as:
\begin{corollary}\label{cor:Intro}
Let $\omega\subset \R^2$ be a bounded Lipschitz domain. Assume that $m,y_0\in C^2(\overline{\omega},\R^3)$ are two regular surfaces and
\begin{align}
\nonumber\mathrm{I}_m(x)=[\D m(x)]^T\D m(x)&=[\D y_0(x)]^T\D y_0(x)=\mathrm{I}_{y_0}(x)\,,\qquad	n(x)=n_0(x)\quad\quad \forall x\in \overline{\omega}\,,\\
m|_{\gamma_d}&=y_0|_{\gamma_d}\,,
\end{align}
where $ n=\frac{\partial_1 m\times \partial_2 m}{\norm{\partial_1 m\times \partial_2 m}}$ and $ n_0=\frac{\partial_1 y_0\times \partial_2 y_0}{\norm{\partial_1 y_0\times \partial_2 y_0}}$ are the respective normal fields and $\gamma_d$ is a relatively open, non-empty subset of the boundary $\partial\omega$.
Then $m\equiv y_0$.
\end{corollary}

The results of this paper should perhaps not come as a surprise to experts in the field of differential geometry. Indeed, except for minimal surfaces, the Gauss map $n_0$ already determines the surface essentially, cf. \cite[Theorem 2.5]{hoffman1983gauss}. On the contrary, minimal surfaces come with a family of `associate surfaces', which have all the same Gauss map but are distinct to each other. Comparable results to Corollary \ref{cor:Intro} can be found in \cite{abe1975isometric,hoffman1983gauss,eschenburg2010compatibility} where different methods of proof were used and any connection to applications were missing. However, the latter result does not belong to the standard textbook knowledge in differential geometry and it is completely unknown in the field of shell-theory. The aim of the present paper is to give a straight forward proof without the techniques coming from differential geometry. We use the Rodrigues representation formula for rotations with given axis as well as repeated properties of the cross-product.

In order to set the stage, we recall some of the better known rigidity and integrability theorems, for $3\text{D}-$bulk materials and $2\text{D}-$surfaces. For a warm up, we tackle first the small rotation problem which already discloses some of the necessary techniques. The paper is now structured as follows: after setting our notations, in the next subsections \ref{moti} and \ref{concept}, we will motivate the problem from an engineering point of view. The reader only interested in the mathematical development can safely skip this part. Then we are in section \ref{rigidity} recapitulating some well-known rigidity and integrability results for bulk materials and surfaces together with simple observations which turn up in our question. We complement this section with some preliminaries on rotations and induced boundary conditions. Then we consider in section \ref{small rot} the case where the rotations in Proposition \ref{propold} are assumed to be small. In the subsequent section \ref{large rot} we give the proof of Proposition \ref{prop1} for the large rotation problem. We end with a conclusion, in section \ref{conclusion}, putting our findings back in an engineering context.

The complementary problem
\begin{equation}
 \D m(x)= U(x)\,\D y_0(x),
\end{equation}
of finding all ``compatible'' in-plane stretches $U$, characterized by
\begin{equation}\label{eq:stretchTensorSzw}
 U(x)\,n_0(x) = \kappa^+(x)\,n_0(x), \quad U(x)\in\Sym^+(3), \quad \kappa^+(x)>0,
\end{equation}
has been completely solved more than ten years ago by Szwabowicz \cite{szwabowicz2008pure}.\footnote{Szwabowicz uses a different notation, but his stretch tensor is basically the stretch tensor $U$ satisfying \eqref{eq:stretchTensorSzw}. Note that since $U\in\Sym^+(3)$ it can be orthogonally diagonalized, the stretch $U(x)$ satisfying \eqref{eq:stretchTensorSzw} leaves the tangent plane $T_{y_0(x)}y_0(\omega)$ invariant. Therefore, the Gauss map is preserved as well.}

\subsection{Notation}
\noindent
For $n\in \mathbb{N}$, we denote the set of all $n\times n$ second order tensors by $\R^{n\times n}$. The identity tensor on $\R^{n\times n}$ is denoted by $\id_n$. As usual we set $\GL(n)= \{ X\in \R^{n\times n}\,|\, \det(X)\neq 0\}$, $\SO(n)=\{X\in \GL(n)\,|\, X^TX=\id_n\,, \det(X)=1\}$ and  $\so(n)=\{X\in \R^{n\times n}\,|\, X^T=-X \}$. The sets $\Sym(n)$ and $\Sym^+(n)$ denote the symmetric and positive definite symmetric tensors respectively. For all $X\in \R^{n\times n}$ we have $\sym X=\frac{1}{2}(X+X^T)\in \Sym(n)$ and $\skew X=\frac{1}{2}(X-X^T)\in\so(n)$.   For $a,b\in \R^n$ we denote the scalar product on $\R^n$ as $\skalarProd{a}{b}$ with associated (squared) norm $\norm{a}^2=\skalarProd{a}{a}$.  We denote $(a|b|c)$ as $3\times 3$ matrix with columns $a,b,c\in\R^3$. For $u\in C^1(\R^3,\R^3)$ we denote the derivative (the Jacobi matrix) of $u$ as $\D u=(\partial_1 u|\partial_2 u|\partial_3 u)$. For $\omega\subset\R^2$ we consider the mapping $m:  \omega\to\R^3$ as a deformation of the midsurface of a shell with Jacobi matrix $\D m = (\partial_1 m|\partial_2 m)$.

The canonical identification of $\so(3)$ and $\R^3$ is denoted by $\axl(\cdot)$. Note that with the standard vector product it is given by $(\axl A)\times \xi =A\,.\,\xi$ for all skew-symmetric matrices $A\in\so(3)$ and vectors $\xi\in \R^3$, so that
\begin{align}
\axl\begin{pmatrix}
0&-v_3&v_2\\v_3&0&-v_1\\-v_2&v_1&0
\end{pmatrix}\coloneqq\begin{pmatrix}
v_1\\v_2\\v_3
\end{pmatrix}\,,\qquad\quad 
\Anti\begin{pmatrix}
v_1\\v_2\\v_3
\end{pmatrix}\coloneqq\begin{pmatrix}
0&-v_3&v_2\\v_3&0&-v_1\\-v_2&v_1&0
\end{pmatrix}\,,
\end{align}
where the inverse of $\axl(\cdot)$ is denoted by $\Anti(\cdot)$. A map $y_0\in C^1(\overline{\omega},\R^3)$ is called a \textit{regular surface} if the Jacobi matrix $\D y_0=(\partial_1 y_0|\partial_2 y_0)$ has rank $2$, i.e., if the surface $y_0(\omega)$ has a well-defined tangent space at all of its (inner) points. Moreover, we make use of the \textit{first fundamental form} $\operatorname{I}_{y_0}$ on $y_0(\omega)$ in matrix representation
\begin{align}
\operatorname{I}_{y_0}:=[\D y_0]^T\D y_0=\begin{pmatrix}
\norm{\partial_1 y_0}^2&\skalarProd{\partial_1 y_0}{\partial_2 y_0}\\\skalarProd{\partial_1 y_0}{\partial_2 y_0}&\norm{\partial_2 y_0}^2
\end{pmatrix}\in \Sym^+(2),
\end{align}
where the positive definitness follows from $\rank{\D y_0}=2$. The matrix representation of the \textit{second fundamental form} on $y_0(\omega)$ is given by
\begin{align}
\operatorname{II}_{y_0}:=-[\D y_0]^T\D n_0=-\begin{pmatrix}
\skalarProd{\partial_1 y_0}{\partial_1 n_0}&\skalarProd{\partial_1 y_0}{\partial_2 n_0}\\\skalarProd{\partial_2 y_0}{\partial_1 n_0}&\skalarProd{\partial_2 y_0}{\partial_2 n_0}
\end{pmatrix}\in \Sym(2),
\end{align}
where the symmetry follows from the fact that $\skalarProd{n_0}{\partial_i y_0}=0$, for $i=1,2$.

\subsection{Engineering motivation: Cosserat shell models}\label{moti}
The elastic range of many engineering materials is restricted to small finite strains. Thin structures may typically undergo large rotations (by bending) but are accompanied by small elastic strains.

\hfill
\includegraphics[width=6cm]{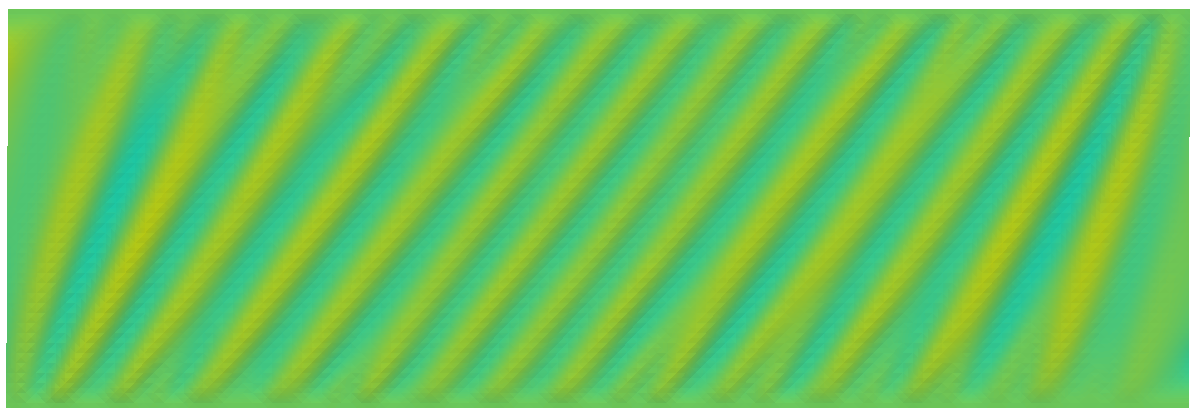}\hfill
\includegraphics[width=6cm]{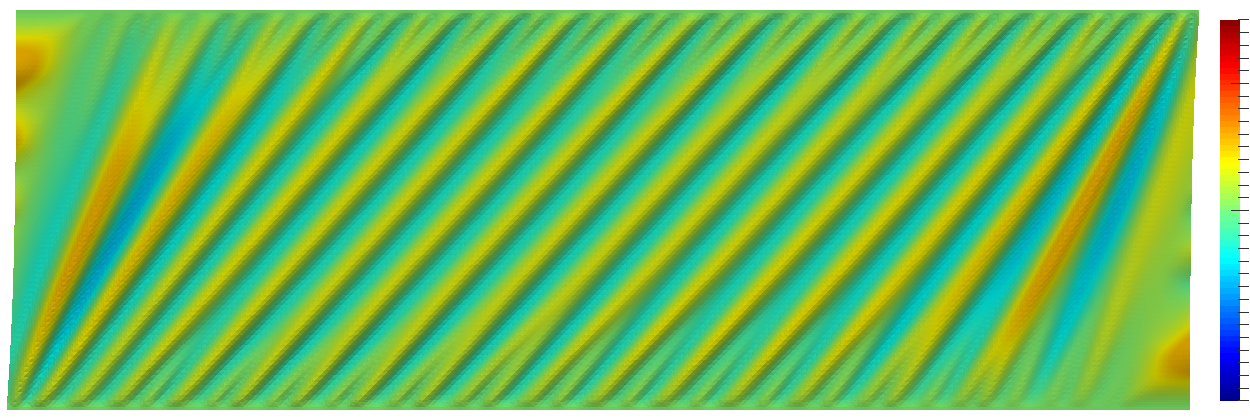}
\hfill{}

\hfill
\includegraphics[width=6cm]{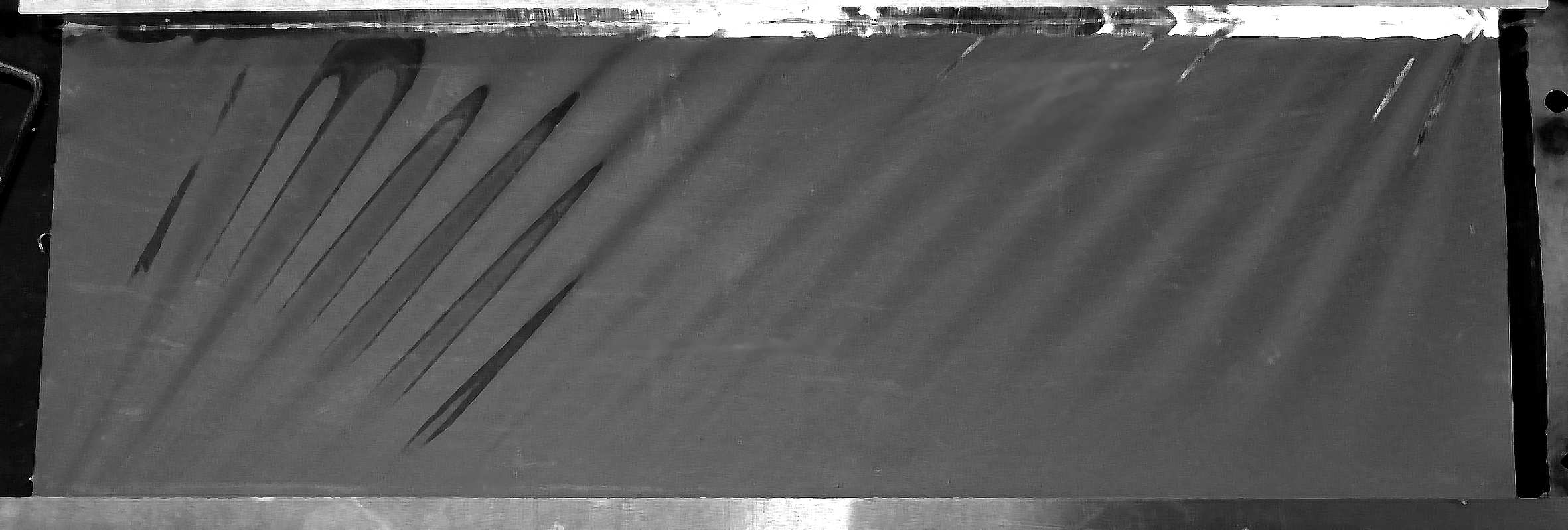}\hfill
\includegraphics[width=6cm]{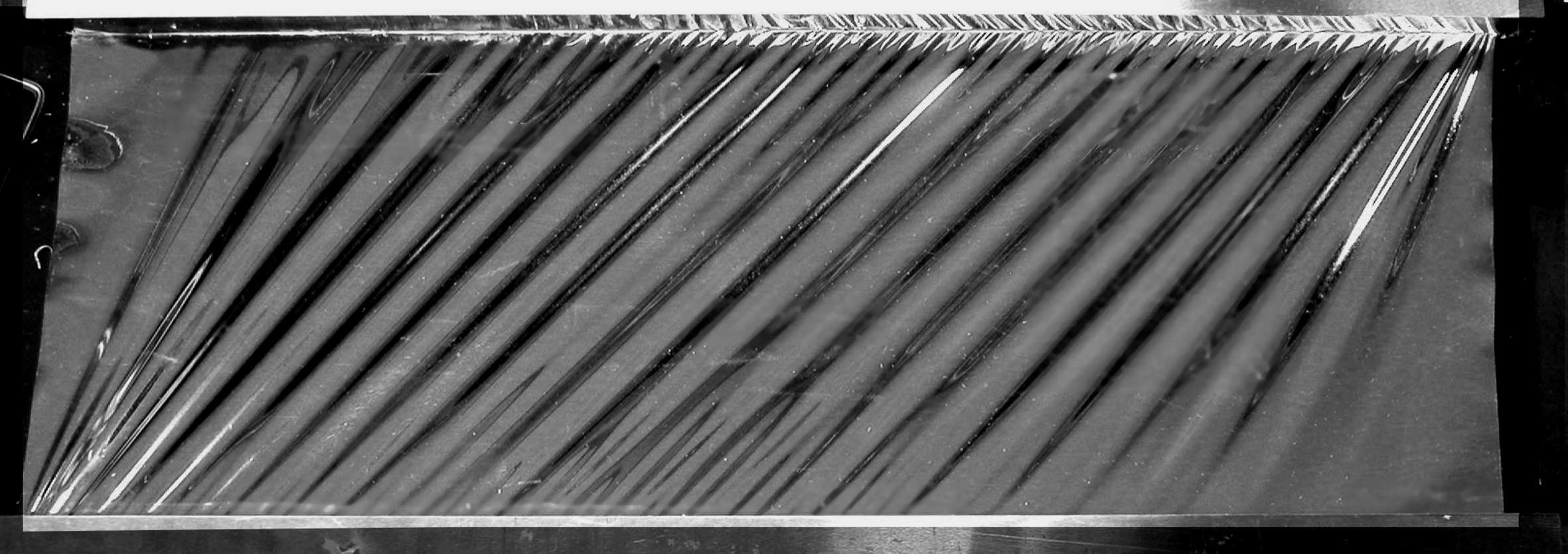}
\hfill{}

In \cite{agn_neff2004geometrically} the following geometrically nonlinear (but small elastic strain) isotropic planar shell model has been derived for such a situation: find the midsurface deformation  $m:  \omega\subset\R^2\to\R^3$ and the independent rotation field $\overline{R}:  \omega\subset \R^2\to \SO(3)$ minimizing the elastic energy 
\begin{align}\label{mimized}
 \nonumber I(m,\overline{R})&=\int_{\omega} h\Big[\mu\underbrace{\norm{\sym((\overline{R}_1|\overline{R}_2)^T\D m-\id_2)}^2}_{\text{shear-stretch energy}}+\mu_c\underbrace{\norm{\skew((\overline{R}_1|\overline{R}_2)^T\D m)}^2}_{\text{first order drill energy}}\\
 &\qquad +\frac{(\mu+\mu_c)}{2}\underbrace{\Big(\iprod{\overline{R}_3,\partial_{1}m}^2+\iprod{\overline{R}_3,\partial_{2}m}^2\Big)}_{\text{transverse shear energy}}+\frac{\mu\lambda}{2\mu+\lambda}\underbrace{\tr(\sym((\overline{R}_1|\overline{R}_2)^T\D m-\id_2))^2}_{\text{ stretch energy}}\Big]\notag\\
 &\qquad + h\Big[\mu L_c^2\norm{\mathcal{K}_s}^2+\mu L_c^{2+q}\norm{\mathcal{K}_s}^{2+q}\Big]\\
\nonumber&\qquad+\frac{h^3}{12}\Big[\mu\norm{\sym\mathcal{K}_b}^2+\mu_c\norm{\skew \mathcal{K}_b}^2+\frac{\mu\lambda}{2\mu+\lambda}\tr[\mathcal{K}_b]^2\Big]\,d\omega \to\min\;\text{w.r.t }(m,\overline{R}),
\end{align}
where the Cosserat curvature tensor is given by
\begin{align}
\mathcal{K}_s=(\overline{R}^T\!\!(\D\,(\overline{R}.e_1)|0),\overline{R}^T\!\!(\D\,(\overline{R}.e_2)|0),\overline{R}^T\!\!(\D\,(\overline{R}.e_3)|0)), \quad\quad \mathcal{K}_b:=\mathcal{K}_{s,3}=\overline{R}^T\!\!(\D\,(\overline{R}.e_3)|0),
\end{align}
and the boundary condition of place for the midsurface deformation $m$ on the Dirichlet part of the lateral boundary, $m|_{\gamma_d}=g_{d}(x,y,0)$ is imposed.
This shell model is derived by dimensional descent from a three-dimensional bulk Cosserat model \cite{cosserat1909theorie,agn_neff2004geometrically} and the appearing parameters are the isotropic shear modulus $\mu>0$, the second Lamé parameter $\lambda$ (with $2\mu+\lambda>0$) and the so-called \textbf{Cosserat couple modulus} $\mu_c\geq 0$, while $h>0$ is the thickness of the shell, $L_c\geq 0$ is a characteristic length and $q\geq 0$. This Cosserat shell model can be naturally related to the general 6-parameter theory of shells \cite{agn_birsan2013characterization,agn_birsan2012equations,agn_birsan2013existence,agn_birsan2014existence,agn_birsan2016dislocation,chroscielewski2011modeling,ibrahimbegovic1994stress,pietraszkiewicz2011refined,pietraszkiewicz2016resultant,wisniewski2000kinematics,wisniewski1998shell}, see also \cite{acharya2020rotations, agn_birsan2014shells, fox1992drill,hughes1989drilling,murin2013consistent,wisniewski2010finite,yeh1993shell}. One of the typical energy terms in these models is connected to so-called \textit{in-plane} drill rotations \cite{agn_birsan2014existence,wisniewski2006enhanced}. These in-plane drill rotations describe local rotations of the shell midsurface with rotation axis given by the local shell normal $n_0$ of $y_0$. Typically, the constitutive coefficients which governs this deformation mode are difficult to establish (and the ubiquitous Cosserat couple modulus $\mu_c>0$ appears prominently). Naghdi-type shell models with only one independent "Cosserat"-director do not have the drill-degree of freedom \cite{ljulj2020naghdi} but allow for transverse shear. Classical shell models neither have drill nor transverse shear \cite{agn_ghiba2020isotropic,eremeyev2006local,pietraszkiewicz2014drilling}. On the contrary, rotations about in-plane axis describe bending and twist. Even though a classical shell model (with Kirchhoff-Love normality assumption) does not have this kinematic degree of freedom, numerical approaches may introduce artificially shell-elements that possess locally this degree of freedom. The question is then which amount of stiffness should be adopted. It is observed that higher artificial in-plane drill stiffness strongly affects the calculated solution. In this context it is also known that flat shell topologies allow for unconstrained drill rotations and we will observe this in our paper as well: indeed unconstrained drill rotations may be observed not only for flat surfaces but for any minimal surface as well.\\
The extension of the planar shell model to initially curved shells has been recently given in \cite{agn_ghiba2020constrained,agn_ghiba2020isotropic,agn_ghiba2020isotropic2}. The planar shell model (\ref{mimized}) has been used to successfully predict the wrinkling behavior of very thin elastic sheets \cite{agn_sander2014numerical}. In these calculations, however, the Cosserat couple modulus $\mu_c$ has been set to zero throughout and $q=2$ has been adopted. In this case, the term in (\ref{mimized}) denoted by "first order drill energy" will drop out, while all other terms basically remain the same. It seems therefore mandatory to devote special attention to this in-plane drill term in order to understand it's physical and mathematical significance. This will be undertaken next.

\subsection{On the physical concept of in-plane drill - linear torsional spring} \label{concept}
Let $\omega\subset \R^2$ be a bounded Lipschitz domain and $y_0:  \omega\subset \R^2\to \R^3$  smooth and regular describing the mid-surface of a shell.
\begin{SCfigure}[][h]
\includegraphics[width=8cm]{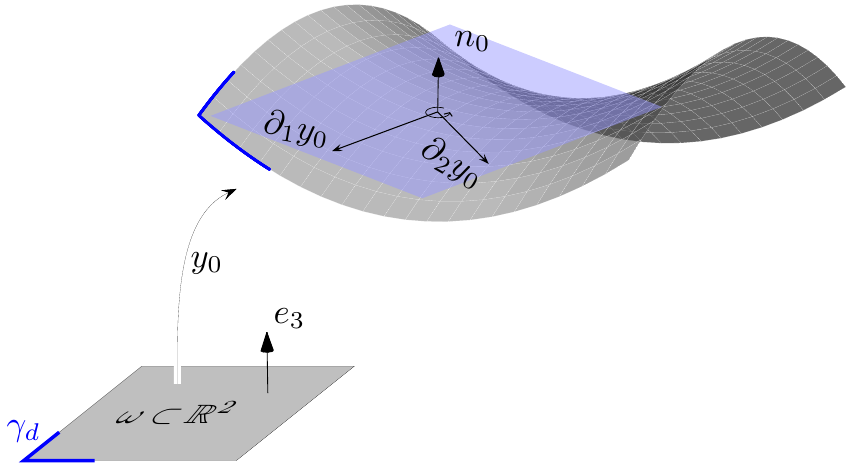}
\caption{The midsurface $y_0\in C^2(\overline{\omega},\R^3)$ of a shell is visualized, together with a tangent plane $T_{y_0(x)}y_0$ spanned by $\partial_1 y_0(x),\partial_2 y_0(x)$ and with unit normal $n_0(x)$. Prescribed boundary conditions at $\gamma_{d}$ fix the shell mid-surface in space.
}
\end{SCfigure}

Let us analyze the energy term corresponding to drill rotations shown in (\ref{mimized}). In order to measure in-plane drill rotations of a shell in a continuum description one first needs to endow the shell with a given orthonormal frame, tangent to the surface $y_0$, against which in-plane rotations can be seen. The role of this frame will be taken here by $Q_0\in \SO(3)$, defined by $Q_0\coloneqq\operatorname{polar}(\D y_0|n_0)$ (already used as such by Darboux, see \cite{darboux1941leccons}), where $\operatorname{polar}(F)$ denotes the orthogonal part in the polar-decomposition of $F\in \GL^+(3)$. First, it holds that
\begin{align}\label{skew}
\skew (Q_0^T(\D y_0|n_0))=0\quad\quad\text{for}\quad\quad Q_0\,=\,\operatorname{polar}(\D y_0|n_0)\in \SO(3),
\end{align}
due to the properties of the polar decomposition \cite{agn_neff2014grioli}
\begin{align}\label{nabla y}
(\D y_0|n_0)&=Q_0\underbrace{
\left(\begin{NiceArray}{CC|C}
 \Block{2-2}{\sqrt{[\D y_0]^T\D y_0}} && 0 \\
 &\hspace*{1.5cm}& 0 \\
 \hline
 0 & 0& 1
 \end{NiceArray}\right)
}_{\in \Sym^+(3)}\notag
\\ \overset{Q_0\,e_3=n_0}{\iff}& \quad (Q_{0_1},Q_{0_2})^T\D y_0 =\sqrt{[\D y_0]^T\D y_0}\in \Sym^+(2)\,. 
\end{align}
Here, it can be seen that $Q_0:  \omega\subset \R^2\to \SO(3)$ is an orthonormal frame whose third column coincides with the normal $n_0$ of the surface such that there is also \textit{no} induced \textit{transverse shear}. The three dimensional condition (\ref{skew}) can be expressed equivalently as (see also (\ref{nabla y}))
\begin{align}\label{skew3}
\nonumber \skew (Q_0^T(\D y_0|n_0))
&=\skew \begin{pmatrix}
\begin{array}{c |c}
(Q_{0_1},Q_{0_2})^T\D y_0 & 0\\\hline
0 & 1
\end{array}
\end{pmatrix}=0\iff \skew[(Q_{0_1}|Q_{0_2})^T\D y_0]=0,
\end{align}
which is the "drill energy" argument from equation (\ref{mimized})  and in this special situation we have an initially "drill-free" setting.

Now, what happens if we only \textbf{locally rotate (drill)} the given tangent vectors $\partial_1 y_0,\partial_2 y_0$ about the rotation axis $n_0$? For this, we take a drill rotation $Q(\alpha)\,n_0=n_0$, $Q(\alpha)=Q(\alpha(x))\in \SO(3)$, where $\alpha=\alpha(x)$ is the rotation angle and $n_0=n_0(x)$ is the prescribed axis of rotation normal to the surface $y_0$ and we consider locally the mapping 
\begin{align}
\D y_0\to Q(\alpha)\D y_0\,,
\end{align}
which leaves the first fundamental form $\operatorname{I}_{y_0}=\D y_0^T\D y_0=\big(Q(\alpha)\D y_0\big)^T\big(Q(\alpha)\D y_0\big)$ invariant. This implies that the surface $y_0$ is locally changed isometrically. For simplicity, taking into account the subsequent representation (\ref{expression}) of rotations with given axis of rotations, we consider presently only small drill rotation angles $\alpha$ so that we can duly approximate 
\begin{align}
Q(\alpha)\approx\id+\alpha\,\Anti(n_0).
\end{align}
Inserting then $(\id+\alpha\,\Anti(n_0))\D y_0$ instead of $\D y_0$ into the drill term from (\ref{skew3}) we obtain
\begin{align}\label{Qskew}
\nonumber\skew& \Big[(Q_{0_1}|Q_{0_2})^T(\id+\alpha\,\Anti(n_0))
\Big(\partial_1 y_0 | \partial_2 y_0\Big)\Big]\\
&=\underbrace{\skew ((Q_{0_1}|Q_{0_2})^T\D y_0)}_{=0}+\skew\Big[(Q_{0_1}|Q_{0_2})^T\,\alpha\,\Big(n_0\times\partial_1 y_0 \Big| n_0\times\partial_2 y_0 \Big)\Big]\\
&=\frac{\alpha}{2}\begin{pmatrix}
0& -[\iprod{Q_{0_2},n_0\times \partial_1 y_0}-\iprod{Q_{0_1},n_0\times \partial_2 y_0}]\\
[\iprod{Q_{0_2},n_0\times \partial_1 y_0}-\iprod{Q_{0_1},n_0\times \partial_2 y_0}]&0
\end{pmatrix}.\notag
\end{align}
We will now show that for any non-zero small angle of rotation $\alpha$, expression (\ref{Qskew}) is non-zero implying that the related drill energy term $h\,\mu_c\norm{\skew((Q_{0_1}|Q_{0_2})^T\D y_0)}^2$ serves to introduce a \textbf{linear torsional spring} \textbf{stiffness} against superposed in-plane rotations (with spring constant $h\mu_c$, where $\mu_c\geq0$ is the Cosserat couple modulus).

Note that at present, the discussion is purely local: at no place did we require that $Q(\alpha(x))\D y_0(x)$ can be determined as the gradient of a mapping. The global question whether $Q(\alpha)\,\D y_0$ can be the gradient of a regular embedding $m:  \omega\to \R^3$ with $\omega\subset\R^2$ will be considered next.

\section{Preliminaries and known rigidity results}\label{rigidity}
\subsection{Setting of the differential geometric problem}
Consider a given initial curved shell surface parametrized locally by $y_0: \overline{\omega}\subset\R^2\to\R^3$, where we assume that $y_0$ is sufficiently smooth and regular ($\rank (\D y_0)=2$). Let $m: \overline{\omega}\subset\R^2\to\R^3$ be any smooth deformation of the given surface $y_0$ parametrized over the same domain and consider a smooth  in-plane drill rotation field 
\begin{align}
Q: \overline{\omega}\subset\R^2\to \SO(3),\qquad Q(x)\,n_0(x)=n_0(x) \,,
\end{align}
where $\displaystyle n_0=\frac{\partial_1 y_0\times\partial_2 y_0}{\norm{\partial_1 y_0\times\partial_2 y_0}}$ is the unit normal vector field on the initial surface $y_0$. We will assume further on that
\begin{align}
Q|_{\gamma_d}=\id,
\end{align}
where  $\gamma_d$ is a relatively open, non-empty subset of the boundary $\partial\omega$. The motivation for this boundary condition will be given in Lemma \ref{lem1}.
\begin{problem}
Let $\omega\subset \R^2$ be a bounded Lipschitz domain and assume that $m,y_0\in C^2(\overline{\omega},\R^3)$ are regular surfaces. Does there exist a nontrivial in-plane drill rotation field $Q\in 
C^1(\overline{\omega},\SO(3))$ such that
\begin{align}\label{condition}
\nonumber\D m(x)&=Q(x)\D y_0(x),\quad\quad Q(x)n_0(x)=n_0(x),\quad\quad x\in \overline{\omega},\\
Q|_{\gamma_d}&=\id.
\end{align}
\end{problem}
\begin{SCfigure}[][h]
\includegraphics[width=9cm]{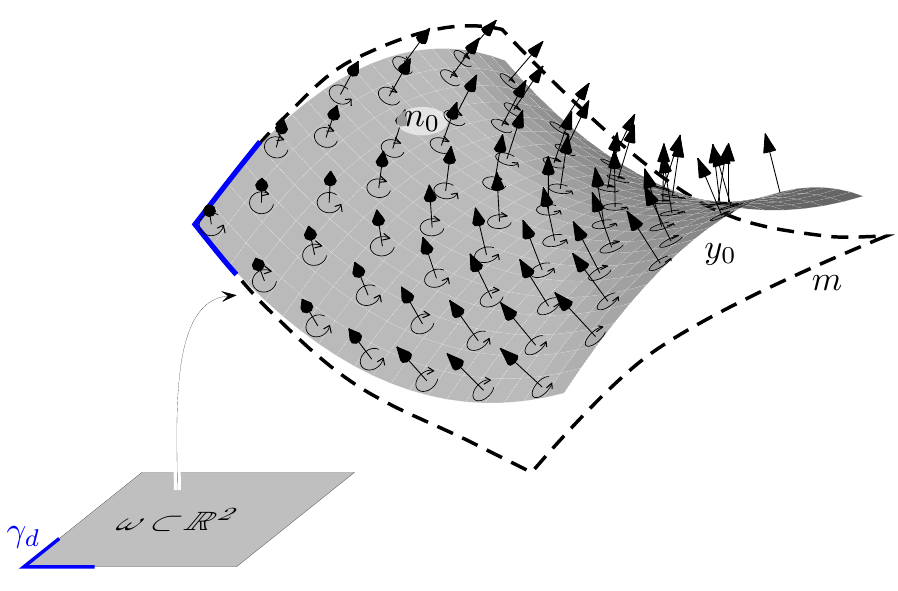}
\caption{At any point of the surface $y_0$ tangent planes are rotated in their own plane leaving the orientation of the surface invariant to obtain a new surface $m$. How can $m$ look like if at a part of the boundary $\gamma_d\subset \partial\omega$ the surfaces $m$ and $y_0$ coincide?}
\end{SCfigure}
\begin{remark}
	For a given shell surface $y_0$ any pure bending (flexure) deformation $m$ satisfies locally
	\begin{align}
	\D m(x)=Q(x)\D y_0(x),\quad\quad Q(x)\in \SO(3),
	\end{align}
	such that the first fundamental forms coincide $\operatorname{I}_m=[\D m]^T\D m=[\D y_0]^T\D y_0=\operatorname{I}_{y_0}$. Considering a flat piece of paper assumed to be made of unstretchable material, the rotations can even be fixed at one side of the paper still allowing for nontrivial bending deformations of the paper. However, the appearing local rotation field $Q(x)\in \SO(3)$  in this case is not of in-plane drill type, i.e., the rotation axis of $Q$ is not everywhere given by $n_0$.
\end{remark}\vspace{-1em}
\begin{SCfigure}[][h]
\includegraphics[width=6cm]{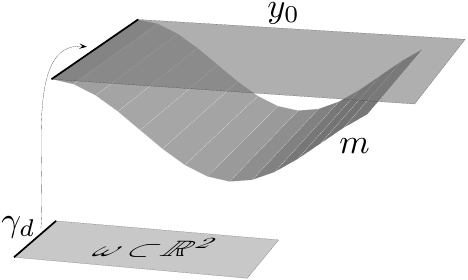}\hspace{2em}
\caption{A pure bending deformation of a surface leaves length invariant, the first fundamental form is unchanged, but there is local rotation. We have $\D m=Q(x)\,\D y_0$ for some non-constant $Q\in \SO(3)$, but the rotation in this example does not have an in-plane rotation axis.}
\end{SCfigure}

\begin{remark}[Rigidity in $3\text{D}$]\label{tim}
	Assume that $M: \Omega\subset\R^3\to \R^3$ and $Y_0: \Omega\subset \R^3\to \R^3$ are two diffeomorphisms, the formally similar to (\ref{condition}) looking condition 
	\begin{align}
	\D M(x)=Q(x)\,\D Y_0(x)\,,\qquad\qquad Q(x)\in \SO(3), \quad\qquad x\in \Omega\,,
	\end{align}
	implies that $Q(x)\equiv\operatorname{const}$ by rigidity \cite{agn_munch2008curl,acharya2000nonlinear}, as can easily be seen as follows.

\noindent
By the chain rule we have 
	\begin{align}
	\nonumber\D\,(M(Y_0^{-1}(\xi)))=\D M(Y_0^{-1}(\xi))\; \D\,[Y_0^{-1}(\xi)]\,, \quad \text{where}\quad
	\D\,[Y_0^{-1}(\xi)]=[\D Y_0(x)]^{-1},
	\end{align}
	therefore,
	\begin{align}
	\nonumber\D\,(M(Y_0^{-1}(\xi)))&=Q(x)[\D Y_0(x)]\;[\D Y_0(x)]^{-1},\\
	 \Rightarrow\quad &\D\,(M(Y_0^{-1}(\xi)))=Q(x)\in \SO(3)\xRightarrow{\text{rigidity}} Q\equiv\text{const.}
	\end{align}
	Note that the smoothness of $Q:\Omega\to \SO(3)$ can be a priori controlled by  the smoothness of $M$ and $Y_0$.
\end{remark}

\begin{remark}[Rigidity in $3\text{D}$]\label{more}
In the $3\text{D}-$case we have another condition which turns out to yield homogeneous rotations as well. Assume again that $M,Y_0: \Omega\subset \R^3\to \R^3$ are two diffeomorphisms. Then
\begin{align}
\forall x\in \Omega:  \quad[\D M(x)]^T\D M(x)=[\D Y_0(x)]^T\D Y_0(x)\quad \iff\quad M(x)=\overline{Q}\,Y_0(x)\,,
\end{align}
where $\overline{Q}\in \SO(3)$ is a constant rotation, as is shown, e.g. in \cite{Ciarlet1988}. However, as already seen, $2\text{D}-$structures are much more flexible in the sense that for $m,y_0: \omega\subset \R^2\to \R^3$ (smooth embeddings)
\begin{align}
\text{I}_m=[\D m(x)]^T\D m(x)&=[\D y_0(x)]^T\D y_0(x)=\text{I}_{y_0}\quad{\centernot\Longrightarrow}\quad m(x)=\overline{Q}\,y_0(x)\,,
\end{align}
as any pure bending deformation shows.
\end{remark}

\begin{example}[Flat case]
	Consider $y_0: \omega\subset \R^2\to \R^3$ with $y_0(x)=(x_1, x_2, 0)^T$. Then $\D y_0(x)=\begin{pmatrix}
	1&0\\0&1\\0&0
	\end{pmatrix}$ and $n_0\equiv e_3$. Thus, the conditions $Q(x)n_0=n_0$, $m\in C^1(\overline{\omega},\R^3)$ and $\D m(x)=Q(x)\,\D y_0(x)$ for $Q(x)\in \SO(3)$, together with $Q|_{\gamma_d}=\id_3$, imply
	\begin{align}
	\D m(x)=Q(x) \begin{pmatrix}
	1&0\\0&1\\0&0
	\end{pmatrix}=
\left(\begin{NiceArray}{CC|C}
 \Block{2-2}{\widehat{Q}(x)} && 0 \\
 &\hspace*{1cm}& 0 \\
 \hline
 0 & 0& 1
 \end{NiceArray}
 \right)\begin{pmatrix}
	1&0\\0&1\\0&0
	\end{pmatrix}
	= \left(\begin{NiceArray}{C}
 \Block{2-1}{\widehat{Q}(x)}\\ 
 \hspace*{1cm}\\
 \hline
 0
 \end{NiceArray}
 \right)\,,
	\end{align}
	with $\widehat{Q}(x)\in \SO(2)$. Hence, $\D m_3(x)\equiv 0$ and
	\begin{align}
	\D\begin{pmatrix}
	m_1(x)\\m_2(x)
	\end{pmatrix}=\widehat{Q}(x)\in \SO(2).
	\end{align} Then, again by rigidity \cite{agn_munch2008curl} we obtain that $\widehat{Q}\equiv\operatorname{const}$. Applying the boundary conditions we have $\widehat{Q}|_{\gamma_d}=\id_2$ and finally $Q\equiv\id_3$. Thus $m-y_0\equiv \operatorname{const}$.
\end{example}

\subsection{Preliminaries on rotations in $\SO(3)$ and the Euler-Rodrigues formula}\label{rotation}
We need to consider the matrix exponential function 
\begin{align}
\exp: \so(3)\to \SO(3), \qquad \exp(A)=\sum_{k=0}^{\infty}\frac{1}{k!}A^k,\quad A\in \so(3).
\end{align}
According to Euler, any rotation can be realized by a rotation around one axis with a certain rotation angle. However, any rotation $Q\in \SO(3)$ with prescribed rotation axis $n_0$ and angle of rotation $\alpha$ can be written with the Euler-Rodrigues representation in matrix notation as \cite{palais2009disorienting}
\begin{align}\label{expression}
\nonumber Q(\alpha)&=\exp(\Anti(\alpha\,n_0))=\id+\sin \alpha\,\Anti(n_0)+(1-\cos \alpha)\,(\Anti(n_0))^2\\
&=(1-\cos\alpha)\,n_0\otimes n_0+\cos\alpha\,\id +\sin\alpha\,\Anti(n_0)\,.
\end{align}

For small rotation angle\,$|\alpha|\!\ll\!1 $\;the above representation \eqref{expression} will be  approximated by 
\begin{align}
Q\approx \id+\alpha\, \Anti(n_0),
\end{align} since ~ $\sin\alpha\to\alpha$ ~ and ~ $1-\cos\alpha\to 0$ ~ \quad as \quad $|\alpha|\to 0$. For later use, note that any nontrivial rotation $\id\neq Q\in \SO(3)$ has (only) one rotation axis $\eta\in \R^3$ such that $Q\,\eta=\eta$ where $\eta$ is the eigenvector to the real eigenvalue $1$.

Taking the trace in \eqref{expression}, we also see that
\begin{align}\label{arc}
\tr(Q(\alpha(x)))=2\cos\alpha(x)+1\qquad \iff \qquad \cos(\alpha(x))=\frac{\tr(Q(\alpha(x)))-1}{2}\,.
\end{align}
The inverse cosine is a multivalued function and each branch is differentiable only on $(-1,1)$. Thus, for $Q=\id$ or $Q=\operatorname{diag}(-1,-1,1)$ we have $\displaystyle\frac{\tr(Q)-1}{2}\in\{-1,1\}$, i.e., in the neighborhood of both these rotations the simple formula \eqref{arc} is not meaningful for extracting a smooth rotation angle.
In order to solve this problem for small rotation angle $\alpha$ (for $Q$ near to $\id$) we proceed as follows (the simple idea is taken from \cite{liang2018efficient}). Multiplying \eqref{expression} on both sides with $\Anti(n_0)$ from the left gives
\begin{align}
\Anti(n_0)Q(\alpha)=\underset{\in\so(3)}{\underbrace{\cos\alpha\,\Anti(n_0)}}+ \sin\alpha\,(\Anti(n_0))^2.
\end{align}
Taking the trace gives
\begin{align}
\tr\big(\Anti(n_0)Q(\alpha)\big)&=- \sin\alpha\,\norm{\Anti(n_0)}^2=-2\sin\alpha\,,
\end{align}
since $\Anti(n_0)\in \so(3)$ and $\norm{n_0}=1$. Thus with \eqref{arc} we arrive at
\begin{align}\label{eq:tan_expression}
\sin\alpha=-\frac{\tr\Big(\Anti(n_0)Q\Big)}{2}\,,\qquad \cos\alpha=\frac{\tr(Q)-1}{2}\,\qquad\overset{\tr(Q)\neq1}{\Longrightarrow}\qquad \tan\alpha=-\frac{\tr\Big(\Anti(n_0)Q\Big)}{\tr(Q)-1}\,,
\end{align}
whereby any branch of the inverse tangent is smooth on $\R$. This shows that for $\tr(Q)-1>0$, i.e., in a large neighborhood of  $Q=\id$, the extraction of the rotation angle $\alpha$ from the rotation $Q$ is as smooth as $Q$ and the surface allows.

\begin{lemma}\label{sing}
	Assume $y\in C^2(\overline{\omega},\R^3)$ is a regular surface and let $Q\in C^1(\overline{\omega},\SO(3))$ be given. Assume that for a point $x_0\in \omega$ and $\alpha_0\in \R$ it holds
	\begin{align}
Q(x_0)=(1-\cos\alpha_0)\,n_0(x_0)\otimes n_0(x_0) +\cos\alpha_0\,\id_3+\sin\alpha_0\,\Anti(n_0(x_0))\,,
	\end{align}	
	where $n_0$ is the normal field on $y_0$. Then there exists a neighborhood $\mathcal{U}(x_0)\subset \omega$ and a continuously-differentiable function 
	\begin{align}
	\alpha:  \mathcal{U}(x_0)\to \R \quad\text{satisfying} \quad \alpha(x_0)=\alpha_0\,,
	\end{align}
	such that for all $x\in \mathcal{U}(x_0)$
	\begin{align}
Q(x)=(1-\cos\alpha(x))\,n_0(x)\otimes n_0(x) +\cos\alpha(x)\,\id_3+\sin\alpha(x)\,\Anti(n_0(x))\,.
	\end{align}
\end{lemma}
\begin{proof}
 The $C^1$-regularity of $\alpha$ in a sufficiently small neighborhood of $x_0$ follows from one of the expressions contained in \eqref{eq:tan_expression}. Indeed, for $\tr(Q(x_0))\neq1$ consider in \eqref{eq:tan_expression}$_3$ the branch of the inverse tangent which contains $\alpha_0$. Otherwise, for $\tr(Q(x_0))=1$ take in \eqref{eq:tan_expression}$_2$ the branch of the inverse cosine which contains $\alpha_0$, cf. Figure \ref{fig:branchtrig}.
\end{proof}

\begin{figure}[h!]
\includegraphics[page=1,width=8cm]{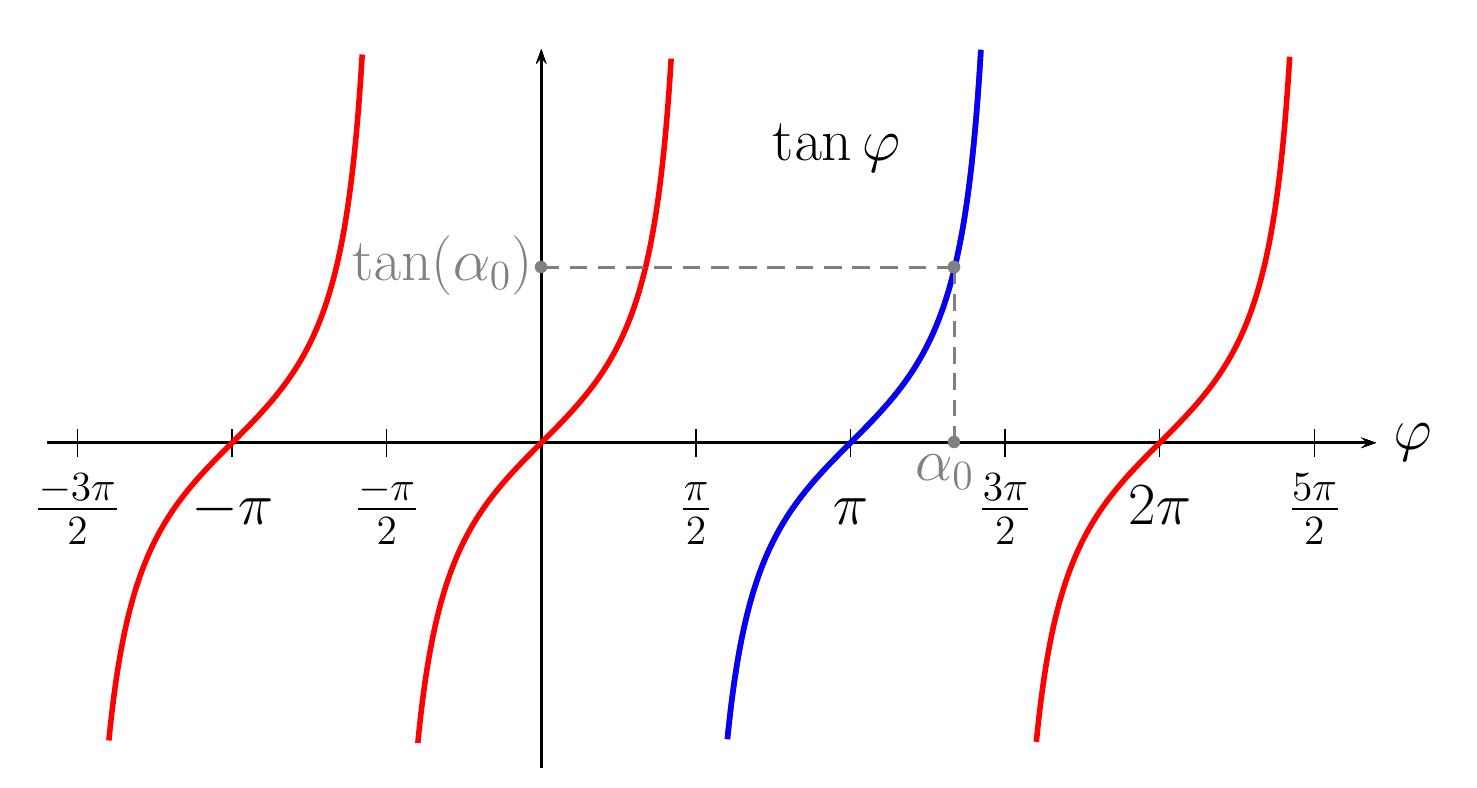}
\includegraphics[page=2,width=8cm]{trigfunc}
\caption{The range for the corresponding branch of the inverse trigonometric functions are indicated in blue.}\label{fig:branchtrig}
\end{figure}

\subsection{Boundary conditions}
\begin{lemma}\label{lem1}
	Let $\omega\subset \R^2$ be a bounded Lipschitz domain. Assume that $m,\,y_0\in C^1(\overline{\omega},\R^3)$ are regular surfaces, $Q\in C^0(\overline{\omega},\SO(3))$ and 
	\begin{align}\label{first}
	\D m(x)=Q(x)\D y_0(x)\,,\quad\quad x\in \overline{\omega}\,,\qquad\quad m|_{\gamma_d}=y_0|_{\gamma_d}\,,
	\end{align}
	where $\gamma_d$ is a relatively open, non-empty subset of the boundary $\partial\omega$.
If for all $x\in \gamma_d$ we have $Q(x)\,n_0(x)=n_0(x)$, where $n_0=\frac{\partial_1y_0\times\partial_2y_0}{\norm{\partial_1y_0\times\partial_2y_0}}$, then $Q|_{\gamma_d}\equiv\id$.
\end{lemma}

\begin{remark}
 For the conclusion of this lemma we only need the assumptions on $\gamma_d$. However, the required conditions in the interior of $\omega$ are those which will be considered later.
\end{remark}

\begin{proof}[Proof of Lemma \ref{lem1}]
	Consider a $C^{0,1}$-parametrization $\gamma:(0,1)\to \R^2$, $\gamma(0,1)\subset \gamma_d\subset\partial \omega$. Then, $m(\gamma(s))=y_0(\gamma(s))$ on $(0,1)$ implies for $\dot{\gamma}(s)\in \R^2$
	\begin{align}\label{second}
	\frac{d}{ds}m(\gamma(s))=\frac{d}{ds}y_0(\gamma(s)) \;\Rightarrow \; \D m(\gamma(s))\,\dot{\gamma}(s)=\D y_0(\gamma(s))\,\dot{\gamma}(s) \;\in \R^3 \qquad \text{a.e. on $(0,1)$.}
	\end{align}
	Hence,
	\begin{align}
	\D m(\gamma(s))\,\,\dot{\gamma}(s)\overset{\eqref{first}}{=}Q(\gamma(s))\D y_0(\gamma(s))\,\,\dot{\gamma}(s)\overset{\eqref{second}}{=}Q(\gamma(s))\underbrace{\D m(\gamma(s))\,\,\dot{\gamma}(s)}_{\eqqcolon q(s)\in \R^3}\qquad \text{a.e. on $(0,1)$.}
	\end{align}
	Thus $q(s)=Q(\gamma(s))\,q(s)$, ~ for almost all $s\in(0,1)$. Moreover, the vector $q(s)$ is a tangent vector to $y_0$ at $y_0(\gamma(s))$. Together with the assumption $Q(\gamma(s))\,n_0(\gamma(s))=n_0(\gamma(s))$ it follows that  $Q(\gamma(s))$ has two linear independent eigenvectors $q(s)$ and $n_0(\gamma(s))$. Since the axis of rotation is unique for any nontrivial rotation, it follows $Q(\gamma(\cdot))=\id$ a.e. on $(0,1)$ and by continuity $Q|_{\gamma_d}\equiv\id$.
\end{proof}

We repeat a similar reasoning for the small rotation case.  
\begin{lemma}\label{first lem}
	Let $\omega\subset \R^2$ be a bounded Lipschitz domain.	Assume that $m,\,y_0\in C^1(\overline{\omega},\R^3)$ are two regular surfaces, $A\in C(\overline{\omega},\so(3))$ and 
	\begin{align}\label{first1}
	\D m(x)=(\id+A(x))\D y_0(x)\,,\quad\quad x\in \overline{\omega}\,,\qquad\quad m|_{\gamma_d}=y_0|_{\gamma_d}\,,
	\end{align}
	where $\gamma_d$ is a relatively open, non-empty subset of the boundary $\partial\omega$.
If for all $x\in \gamma_d\subset {\partial\omega}$ we have $A(x)\,n_0(x)=0$, where $n_0=\frac{\partial_1y_0\times\partial_2y_0}{\norm{\partial_1y_0\times\partial_2y_0}}$, then $A|_{\gamma_d}\equiv 0$.
\end{lemma}
\begin{proof}
We have again
	\begin{align}\label{second1}
	\D m(\gamma(s))\,\dot{\gamma}(s)=\D y_0(\gamma(s))\,\dot{\gamma}(s) \;\in \R^3 \qquad \text{a.e. on $(0,1)$}\,,
	\end{align}
for a $C^{0,1}$-parametrization $\gamma:(0,1)\to \R^2$, $\gamma(0,1)\subset \gamma_d\subset\partial \omega$. Hence, we obtain a.e. on $(0,1)$
\begin{align}
\nonumber	\D y_0(\gamma(s))\dot{\gamma}(s)&\overset{\eqref{second1}}{=}\D m(\gamma(s))\,\,\dot{\gamma}(s)\overset{\eqref{first1}}{=}\D y_0(\gamma(s))\,\,\dot{\gamma}(s)+A(\gamma(s))\,\D y_0(\gamma(s))\,\,\dot{\gamma}(s)\\
 \qquad \iff \qquad 0\,&= A(\gamma(s))\underbrace{\D y_0(\gamma(s))\,\,\dot{\gamma}(s)}_{\eqqcolon\,q(s)}=A(\gamma(s))\,q(s)\,,\quad \text{where} \quad q(s)\bot n_0(\gamma(s))\,.
	\end{align}
	From the assumption we moreover have $A(\gamma(s))\,n_0(\gamma(s))=0$ so that with $A(\gamma(s))\,q(s)=0$ along $\gamma_d$ we obtain $A(\gamma(s))=0$, since any non-zero skew-symmetric $3\times 3$ matrix $A$ has rank two.
\end{proof}	

\begin{proposition}[Compatibility for surfaces]\label{newcom} Let $v,w\in C^1(\omega,\R^3)$ be given vector fields with $\rank(v,w)=2$ everywhere and assume that $\omega$ is a bounded, simply connected domain. Then there exists a regular surface $m\in C^2(\omega,\R^3)$ such that
\begin{align}\label{comcon}
\D m(x)=\Big(v(x)\big|w(x)\Big)\,,\qquad x\in \omega\,,
\end{align}
if and only if the \textbf{compatibility condition}
\begin{align}
\partial_2 v(x_1,x_2)=\partial_1 w(x_1,x_2)\,,
\end{align}
holds. The surface $m$ is unique up to translation.
\end{proposition}
\begin{proof}
	We only need to observe that \eqref{comcon} reads
\begin{align}
(\partial_1 m \ | \ \partial_2 m) = (v \ | \ w)
\end{align}
and the existence of potential $m\in C^2(\omega,\R^3)$ is guaranteed if and only if 
\begin{align}
\partial_2  v=\partial_2\partial_1 m \overset{m\in C^2}{=}\partial_1\partial_2 m =\partial_1 w,
\end{align}
where the potential is unique up to a constant.
\end{proof}

\section{The small rotation case: $A\in \so(3)$}\label{small rot}
For $m,y_0: \omega\subset\R^2\to\R^3$ we discuss
\begin{align}
\D m(x)=Q(x)\,\D y_0(x)\,,\quad Q(x)\,n_0(x)=n_0(x)\,,\quad n_0\;\;\text{unit normal vector field on } y_0(\omega)\,,
\end{align}
where $Q\in C^1(\omega,\text{SO}(3))$. Consider a linear approximation of this situation for small rotation angle $\alpha$, $|\alpha|\ll1$. Then we can write
\begin{align}
&\D m(x)=\D y_0(x)+\D v(x),\quad Q(x)=\id+A(x)+\text{h.o.t.}\,, \quad A\in C^1(\omega, \so(3)). 
\end{align}
Note that we do not assume that $\mathrm\bf{\D v}$ is small. We only assume that the rotations are close to $\id$. Then,
\begin{align}
\D m(x)=\D y_0(x)+\D v(x)=(\id+A(x)+\ldots)\D y_0(x)\qquad\Rightarrow\qquad\D v(x)=A(x)\D y_0(x)+\ldots\,,
\end{align}
hence we may consider the new problem 
\begin{align}\label{nabla}
\D v(x)=A(x)\D y_0(x)\,,\qquad\quad A\in C^1(\omega,\so(3))\,,\qquad A(x)\,n_0(x)=0\,.
\end{align}
Therefore, we are led to study the problem

\begin{lemma}\label{LemAngle1}
	Let $\omega\subset\R^2$ be a bounded Lipschitz domain. Assume $y_0\in C^2(\omega,\R^3)$ is a regular surface and let $v\in C^2(\omega,\R^3)$. Moreover, assume $\alpha\in C^1(\omega,\R)$ and consider the system
		\begin{align}\label{exp v}
	\D v(x)&=\alpha(x)\;\Anti(n_0(x))\,\D y_0(x)\,,\quad\quad x\in \omega\,,
	\end{align}
	where $n_0=\frac{\partial_1 y_0\times \partial_2 y_0}{\norm{\partial_1 y_0\times \partial_2 y_0}}$ denotes the normal field on $y_0(\omega)$. Then $\alpha\equiv\operatorname{const}$.
\end{lemma}

\begin{proof}
 We write
\begin{align}
(\partial_1 v | \partial_2 v)&=\alpha\Anti(n_0)(\partial_1  y_0 |\partial_2  y_0)=\alpha\Big(\Anti(n_0)\partial_1  y_0\,\Big| \Anti(n_0)\partial_2  y_0\Big)\,,\\
\nonumber
\Longleftrightarrow \quad  \partial_1 v &=\alpha\,[\,n_0\times \partial_1 y_0\,]\,,\quad\partial_2 v=\alpha\,[\,n_0\times \partial_2 y_0\,]\,,
\end{align}
where we have used that $\Anti(n_0)\eta=n_0\times \eta$. We proceed by taking the mixed derivatives
\begin{align}\label{mixed}
\nonumber \partial_2 \partial_1 v&=\partial_2 \alpha\,[\,n_0\times \partial_1 y_0\,]+\alpha\,[\,\partial_2 n_0\times \partial_1 y_0+n_0\times \partial_2 \partial_1 y_0\,],\\
\partial_1 \partial_2 v&=\partial_1 \alpha\,[\,n_0\times \partial_2 y_0\,]+\alpha\,[\,\partial_1 n_0\times \partial_2 y_0+n_0\times \partial_1 \partial_2 y_0\,].
\end{align}
Hence, by equality of the mixed derivatives in \eqref{mixed} for $y_0,v\in C^2(\omega,\R^3)$ we must have
\begin{align}\label{A,B}
\partial_2 \alpha\underbrace{[\,n_0\times \partial_1 y_0\,]}_{\eqqcolon\vv*{Y}{0}}+\alpha\underbrace{[\,\partial_2 n_0\times \partial_1 y_0\,]}_{\eqqcolon\vv{B}}=\partial_1 \alpha\underbrace{[\,n_0\times \partial_2 y_0\,]}_{\eqqcolon-\vv*{X}{0}}+\alpha\underbrace{[\,\partial_1 n_0\times\partial_2  y_0\,]}_{\eqqcolon\vv{A}}\in \R^3\,.
\end{align}
Especially we have that $\vv{A}$ and $\vv{B}$ are normal vectors whereas $\vv*{X}{0}$ and $\vv*{Y}{0}$ are linear independent tangent vectors, since ~$\skalarProd{\vv*{X}{0}}{n_0} =0$, ~ $\skalarProd{\vv*{Y}{0}}{n_0} =0$ ~ and
\begin{align}\label{eq:X0Y0}
 \vv*{X}{0}\times \vv*{Y}{0} & = -(n_0\times \partial_2  y_0)\times (n_0\times \partial_1 y_0) = - \skalarProd{n_0}{\partial_2y_0\times\partial_1y_0}n_0 = \skalarProd{n_0}{\norm{\partial_1 y_0\times \partial_2 y_0}\cdot n_0}n_0\notag\\
 & = \norm{\partial_1 y_0\times \partial_2 y_0}\cdot n_0 = \partial_1 y_0\times \partial_2 y_0,
\end{align}
where we have used that  $n_0=\frac{\partial_1 y_0\times \partial_2 y_0}{\norm{\partial_1 y_0\times \partial_2 y_0}}$. Thus, the vector fields $\vv*{X}{0}$,  $\vv*{Y}{0}$ and $n_0$ form a $3-$frame on the surface $y_0(\omega)$. However, \eqref{A,B} reads,
\begin{align}\label{eq:linindepent}
 \partial_1\alpha \cdot\vv*{X}{0} + \partial_2\alpha\cdot \vv*{Y}{0} = \alpha\cdot(\vv{A}-\vv{B})= \delta\cdot n_0,
\end{align}
with a scalar field $\delta$, so that by the linear independence of the vector fields  $\vv*{X}{0}$,  $\vv*{Y}{0}$ and $n_0$ we must always have $\partial_1\alpha=\partial_2\alpha=\delta=0$, which gives $\alpha\equiv\operatorname{const}$.
\end{proof}
Thus, adding sufficient boundary conditions we arrive at

\begin{proposition}\label{propold}
	Let $\omega\subset\R^2$ be a bounded Lipschitz domain. Assume $y_0\in C^2(\omega,\R^3)$ is a regular surface and let $v\in C^2(\omega,\R^3)$. Moreover assume $\alpha\in C^1(\omega,\R) \cap C^0(\overline{\omega},\R)$ and consider the system
		\begin{align}\label{exp v1}
	\D v(x)&=\alpha(x)\;\Anti(n_0(x))\,\D y_0(x)\,,\quad\quad x\in \omega\,,\quad\quad\alpha|_{\gamma_d}=0\,,
	\end{align}
	where $n_0=\frac{\partial_1 y_0\times \partial_2 y_0}{\norm{\partial_1 y_0\times \partial_2 y_0}}$ denotes the normal field on $y_0(\omega)$ and $\gamma_d$ is a relatively open, non-empty, subset of the boundary $\partial\omega$. Then $\alpha\equiv0$.
\end{proposition}
\begin{proof}
 By Lemma \ref{LemAngle1} it holds $\alpha\equiv\operatorname{const}$, so that due to the vanishing boundary condition $\alpha|_{\gamma_d}=0$ and the continuity of $\alpha$ we obtain $\alpha\equiv0$.
\end{proof}

\begin{corollary}\label{corlin}
Let $\omega\subset \R^2$ be a bounded Lipschitz domain. Assume that $m,\,y_0\in C^2(\overline{\omega},\R^3)$ are regular surfaces and $\alpha\in C^1(\omega,\R)\cap C^0(\overline{\omega},\R )$ is given with
\begin{align}\label{eq:Bedm}
\D m(x)=(\id+\alpha(x)\Anti(n_0(x)))\,\D y_0(x)\,,\quad\quad x\in \overline{\omega}\,,\quad\quad m|_{\gamma_d}=y_0|_{\gamma_d}\,,
\end{align}
where $n_0=\frac{\partial_1 y_0\times \partial_2 y_0}{\norm{\partial_1 y_0\times \partial_2 y_0}}$ denotes the normal field on $y_0(\omega)$ and $\gamma_d$ is a relatively open, non-empty, subset of the boundary $\partial\omega$. Then $m\equiv y_0$.	
\end{corollary}
\begin{proof}
	We invoke Lemma \ref{first lem} to see that $m|_{\gamma_d}=y_0|_{\gamma_d}$ implies $\alpha|_{\gamma_d}\equiv 0$. Thus, for $v=m-y_0$ we can apply Proposition \ref{propold} to conclude that $\D v\equiv 0$.
\end{proof}

\begin{proposition}\label{prop:meanCurvature1}
 	Let $\omega\subset\R^2$ be a bounded Lipschitz domain. Assume $m,y_0\in C^2(\omega,\R^3)$ are regular surfaces and $\alpha\in C^1(\omega,\R)$ with 
		\begin{align}\label{exp v12}
	\D m(x)&=(\id+\alpha(x)\Anti(n_0(x)))\D y_0(x)\,,\quad\quad x\in \omega\,,
	\end{align}
	where $n_0=\frac{\partial_1 y_0\times \partial_2 y_0}{\norm{\partial_1 y_0\times \partial_2 y_0}}$ denotes the normal field on $y_0(\omega)$. Then 
	\begin{equation}\label{eq:conclH}
	 \forall \ x\in \omega: \quad \alpha(x)=0 \quad \text{or}\quad \EuScript{H}(x)=0,
	\end{equation}
where $\EuScript{H}$ denotes the \textit{mean curvature} on the surface $y_0$.
\end{proposition}
\begin{proof}
 Recall that it holds for the vector field
 \begin{equation}\label{eq:meanCurvature}
  \vv{A}-\vv{B} = \partial_1 n_0\times\partial_2  y_0 - \partial_2 n_0\times \partial_1 y_0 =-2\,\EuScript{H}\,\norm{\partial_1 y_0\times \partial_2 y_0}\,n_0=-2\EuScript{H}\,\partial_1 y_0\times \partial_2 y_0,
 \end{equation}
cf. \cite[Section 2.5, Theorem 2]{minimalSurfaces339}. Thus, for $v=m-y_0$ the validity of \eqref{eq:linindepent} implies that we have pointwise either a vanishing angle $\alpha$ or a vanishing mean curvature $\EuScript{H}$ since
\begin{equation*}
0\underset{\alpha\equiv\operatorname{const}}{\overset{\eqref{eq:linindepent}}{\equiv}}\alpha\cdot(\vv{A}-\vv{B})\overset{\eqref{eq:meanCurvature}}{=}-2\,\alpha\cdot\EuScript{H}\cdot\partial_1 y_0\times \partial_2 y_0\,.\qedhere
\end{equation*}
\end{proof}

\begin{remark}[Symmetry of the second fundamental form]
 It is interesting to note that conclusion \eqref{eq:conclH} can also be obtained from the symmetry property of the second fundamental form on the surface $m(\omega)$. Indeed, the normal vector field on $m(\omega)$ coincides with $n_0$ since
 \begin{align}
  \partial_1 m \times \partial_2 m \ &\overset{\mathclap{\eqref{exp v12}}}{=}\ (\partial_1y_0+\alpha\,\Anti(n_0)\, \partial_1 y_0)\times(\partial_2y_0 +\alpha\,\Anti(n_0)\, \partial_2 y_0)\notag\\
   & =  \partial_1 y_0 \times \partial_2 y_0 + \alpha^2(n_0\times \partial_1 y_0)\times(n_0\times \partial_2 y_0) \overset{\eqref{eq:X0Y0}}{=}(1+\alpha^2)\,\partial_1 y_0 \times \partial_2 y_0.
 \end{align}
Thus, for the second fundamental form on $m(\omega)$ we obtain
\begin{align}
 \Sym(2)\ni \mathrm{II}_m&= -[\D m]^T\D n= -[\D m]^T\D n_0 \overset{\eqref{exp v12}}{=}-[(\id+\alpha\Anti(n_0))\D y_0]^T \D n_0 \\
 & = -[\D y_0]^T \D n_0-\alpha\,[\D y_0]^T \Anti(n_0)^T\D n_0 = \mathrm{II}_{y_0}+\alpha\,[\D y_0]^T \Anti(n_0)\D n_0. \notag
\end{align}
Since $ \mathrm{II}_{y_0}\in\Sym(2)$, we are left with the single condition
\begin{align}\label{eq:ersterSchritt1}
 \alpha\,[\D y_0]^T \Anti(n_0)\D n_0 \in \Sym(2) \quad 
  &\Longleftrightarrow \quad \alpha \begin{pmatrix} * & \skalarProd{\partial_1 y_0}{n_0\times\partial_2 n_0} \\ \skalarProd{\partial_2 y_0}{n_0\times\partial_1 n_0} & * \end{pmatrix}\in\Sym(2) \notag\\
 &\Longleftrightarrow \quad \{\ \alpha = 0 \quad \text{or} \quad \skalarProd{n_0}{\partial_2 n_0\times\partial_1y_0}= \skalarProd{n_0}{\partial_1 n_0\times\partial_2 y_0}\} \notag \\
 &\Longleftrightarrow \quad \{\ \alpha = 0 \quad \text{or} \quad 0=\skalarProd{n_0}{\partial_1 n_0\times\partial_2 y_0-\partial_2 n_0\times\partial_1y_0}\} \notag \\
 &\overset{\mathclap{\eqref{eq:meanCurvature}}}{\Longleftrightarrow} \quad \{\ \alpha = 0 \quad \text{or} \quad 0=\skalarProd{n_0}{-2\EuScript{H}\,\partial_1 y_0\times \partial_2 y_0}\} \notag \\
 &\Longleftrightarrow \quad \{\ \alpha = 0 \quad \text{or} \quad \EuScript{H}\norm{\partial_1 y_0\times \partial_2 y_0}=0\} \notag\\
 & \Longleftrightarrow \quad \{\ \alpha = 0 \quad \text{or} \quad \EuScript{H}=0\}.
\end{align}

\end{remark}

\begin{corollary}
 	Let $\omega\subset\R^2$ be a bounded Lipschitz domain. Assume $y_0\in C^2(\omega,\R^3)$ is a regular surface and let $v\in C^2(\omega,\R^3)$. Moreover assume $\alpha\in C^1(\omega,\R)$ and consider the system
		\begin{align}\label{exp v0}
	\D v(x)&=\alpha(x)\,\Anti(n_0(x))\,\D y_0(x)\,,\quad\quad x\in \omega\,,
	\end{align}
	where $n_0=\frac{\partial_1 y_0\times \partial_2 y_0}{\norm{\partial_1 y_0\times \partial_2 y_0}}$ denotes the normal field on $y_0(\omega)$. If the mean curvature $\EuScript{H}$   of $y_0$ does not vanish at one point $y_0(x_0)$, then $\alpha\equiv0$.
\end{corollary}
\begin{proof}
 It follows from the previous Proposition \ref{prop:meanCurvature1}, that if the mean curvature $\EuScript{H}$ does not vanish at some point $y_0(x_0)$, we must have $\alpha(x_0)=0$ and the conclusion follows, since $\alpha\equiv\operatorname{const}$ by Lemma \ref{LemAngle1}.
\end{proof}

\section{The large rotation case: $Q\in \SO(3)$}\label{large rot}
\begin{lemma}\label{LemAngle2}
 Let $\omega\subset \R^2$ be a bounded Lipschitz domain.	Assume that $m,\,y_0\in C^2(\omega,\R^3)$ are regular surfaces, $Q\in C^1(\omega,\SO(3))$ and 
	\begin{align}\label{origin3424}
	\D m(x)&=Q(x)\D y_0(x)\,,\quad\quad Q(x)\,n_0(x)=n_0(x)\,,\quad\quad x\in \omega\,,
	\end{align}
	where $n_0=\frac{\partial_1 y_0\times \partial_2 y_0}{\norm{\partial_1 y_0\times \partial_2 y_0}}$ denotes the normal field on $y_0(\omega)$.  Then $\cos\alpha\equiv\operatorname{const}$ and $\sin\alpha\equiv\operatorname{const}$, where $\alpha:\omega\to\R$ is any function denoting the rotation angle in the Euler-Rodrigues representation of $Q$.
\end{lemma}
\begin{proof}
 By the Euler-Rodrigues representation there exists a function $\alpha:\omega\to\R$ which fulfills
	\begin{align}\label{Q}
	Q(x)=(1-\cos \alpha(x))\,n_0(x)\otimes n_0(x)+\cos\alpha(x)\,\id+\sin\alpha(x)\,\Anti(n_0(x))\,.
	\end{align}
However, the rotation angle is uniquely defined up to the additive term $2k\pi$, $\pi\in\mathbb{Z}$. Thus, $\alpha$ even cannot assumed to be continuous. But we have already seen in the expressions \eqref{eq:tan_expression} that, at least locally, the choice of the rotation angle can be made sufficiently smooth, i.e.,  for all ~$x\in\omega$~ there exists a neighborhood $U$ of $x$  and a function $\widetilde{\alpha}\in C^1(U\cap \omega,\R)$ which denotes the rotation angle in the Euler-Rodrigues formula.

	In view of \eqref{Q}, the problem \eqref{origin3424} can be recast (at least locally) as follows
	\begin{align}
	\D m=\Big((1-\cos \widetilde{\alpha})\,n_0\otimes n_0+\cos\widetilde{\alpha}\,\id+\sin\widetilde{\alpha}\,\Anti(n_0)\Big)\D y_0.
	\end{align}
	Since $(n_0\otimes n_0)\D y_0=n_0\otimes ([\D y_0]^Tn_0)=0$, the latter formula simplifies to
	\begin{align}
	\D m=\cos\widetilde{\alpha}\,\D y_0+\sin\widetilde{\alpha}\,\Anti(n_0)\,\D y_0.
	\end{align}
	Obviously we have
	\begin{align}\label{m}
	\nonumber\partial_1 m=\cos\widetilde{\alpha}\,\partial_1 y_0+\sin\widetilde{\alpha}\,(n_0\times \partial_1 y_0) \quad 
	\text{and}\quad \partial_2 m=\cos\widetilde{\alpha}\,\partial_2 y_0+\sin\widetilde{\alpha}\,(n_0\times \partial_2 y_0)\,,
	\end{align}
	where for $i=1,2$, we used that $\Anti(n_0)\partial_i\,y_0=n_0\times \partial_iy_0$.
	By taking the mixed derivatives, we arrive at
	\begin{align}
&\partial_2 \partial_1 m=\partial_2 (\cos\widetilde{\alpha} \,\,\partial_1 y_0)+\partial_2 (\sin\widetilde{\alpha}\,(n_0\times \partial_1 y_0))\\
&=-\sin\widetilde{\alpha}\,\partial_2 \widetilde{\alpha}\,\partial_1 y_0+\cos\widetilde{\alpha}\,\partial_2 \partial_1 y_0+\cos\widetilde{\alpha}\,\partial_2 \widetilde{\alpha}\,(n_0\times \partial_1 y_0)+\sin\widetilde{\alpha}\,\partial_2 n_0\times \partial_1 y_0 +\sin\widetilde{\alpha}\,n_0\times \partial_2 \partial_1 y_0, \notag\\
\shortintertext{as well as}
	&\partial_1 \partial_2 m=\partial_1 (\cos\widetilde{\alpha}\,\,\partial_2 y_0)+\partial_1 (\sin\widetilde{\alpha}\,(n_0\times \partial_2 y_0))\\
	&=-\sin\widetilde{\alpha}\,\partial_1 \widetilde{\alpha}\,\partial_2 y_0+\cos\widetilde{\alpha}\;\partial_1 \partial_2 y_0+\cos\widetilde{\alpha}\,\partial_1 \widetilde{\alpha}\,(n_0\times \partial_2 y_0)+\sin\widetilde{\alpha}\;\partial_1 n_0\times \partial_2 y_0+\sin\widetilde{\alpha}\,n_0\times \partial_1 \partial_2 y_0. \notag
	\end{align}
	By using the equality of mixed derivatives for $m,y_0\in C^2(\omega,\R^3)$ we must have
	\begin{align}\label{eq:A,B2}
	\partial_2 \widetilde{\alpha}\underbrace{\Big(-\sin\widetilde{\alpha}\;\partial_1 y_0+\cos\widetilde{\alpha}\,(n_0\times \partial_1 y_0)\Big)}_{\eqqcolon\vv*{Y}{\widetilde{\alpha}}}&+\sin\widetilde{\alpha}\;\underbrace{\partial_2 n_0\times \partial_1 y_0}_{=\vv{B}}\\
\nonumber	&= \!\partial_1 \widetilde{\alpha}\underbrace{\Big(\!-\sin\widetilde{\alpha}\,\partial_2 y_0+\!\cos\widetilde{\alpha}(n_0\times\!\partial_2 y_0)\Big)}_{\eqqcolon-\vv*{X}{\widetilde{\alpha}}}+\!\sin\widetilde{\alpha}\;\underbrace{\!\partial_1 n_0\times \partial_2 y_0}_{=\vv{A}}.
	\end{align}
As in the linear case we obtain that $\vv{A}$ and $\vv{B}$ are normal vectors whereas $\vv*{X}{\widetilde{\alpha}}$ and $\vv*{Y}{\widetilde{\alpha}}$ are linear independent tangent vectors, since ~$\skalarProd{\vv*{X}{\widetilde{\alpha}}}{n_0} =0$, ~ $\skalarProd{\vv*{Y}{\widetilde{\alpha}}}{n_0} =0$ ~ and
\begin{align}
 \vv*{X}{\widetilde{\alpha}}\times \vv*{Y}{\widetilde{\alpha}} & = (\sin\widetilde{\alpha}\, \partial_2 y_0+\cos\widetilde{\alpha}\, \vv*{X}{0})\times(-\sin\widetilde{\alpha}\,\partial_1y_0+\cos\widetilde{\alpha}\,\vv*{Y}{0}) \notag\\
 & = \sin^2\widetilde{\alpha}\, \partial_1 y_0\times\partial_2 y_0 + \cos^2\widetilde{\alpha}\,  \vv*{X}{0}\times \vv*{Y}{0} + \sin\widetilde{\alpha}\cos\widetilde{\alpha}(\partial_2 y_0 \times \vv*{Y}{0} - \vv*{X}{0}\times\partial_1 y_0) \notag\\
  & \underset{\ast}{\overset{\mathclap{\eqref{eq:X0Y0}}}{=}} \ \sin^2\widetilde{\alpha}\,\partial_1 y_0\times\partial_2 y_0 +  \cos^2\widetilde{\alpha}\, \partial_1 y_0\times\partial_2 y_0 = \partial_1 y_0\times\partial_2 y_0,
\end{align}
where in $\ast$ we also used that $$ \partial_2 y_0 \times \vv*{Y}{0} - \vv*{X}{0}\times\partial_1 y_0 = \partial_2 y_0 \times (n_0\times\partial_1y_0)+(n_0\times\partial_2y_0)\times\partial_1y_0 = 0.$$

Thus, the vector fields $\vv*{X}{\widetilde{\alpha}}$,  $\vv*{Y}{\widetilde{\alpha}}$ and $n_0$ form a $3-$frame on the surface $y_0(\omega)$. However, \eqref{eq:A,B2} reads,
\begin{align}\label{eq:linindepentalpha}
 \partial_1\widetilde{\alpha} \cdot\vv*{X}{\widetilde{\alpha}} + \partial_2\widetilde{\alpha}\cdot \vv*{Y}{\widetilde{\alpha}} = \sin\widetilde{\alpha}\cdot(\vv{A}-\vv{B})= \widetilde\delta\cdot n_0,
\end{align}
with a scalar field $\widetilde\delta$, so that by the linear independence of the vector fields  $\vv*{X}{\widetilde{\alpha}}$,  $\vv*{Y}{\widetilde{\alpha}}$ and $n_0$ we  must always have $\partial_1\widetilde{\alpha}=\partial_2\widetilde{\alpha}=\widetilde\delta=0$,  which gives $\widetilde{\alpha}\equiv\operatorname{const}$ in $U\cap \omega$ and also $\cos\widetilde{\alpha}\equiv\operatorname{const}$ as well as $\sin\widetilde{\alpha}\equiv\operatorname{const}$ in $U\cap \omega$. Consequently, we also have $\cos\alpha\equiv\operatorname{const}$ and $\sin\alpha\equiv\operatorname{const}$ in $U\cap \omega$ for any choice of an angle function $\alpha:\omega\to\R$. It follows from the expressions in \eqref{eq:tan_expression} that both functions $\cos\alpha(x)$ and $\sin\alpha(x)$ are continuous which yields the conclusion.
\end{proof}

\begin{remark}\label{rem:constantangle}
 Note that the rotation angle $\alpha$ can always be constant but unconstrained, if, e.g., no further boundary conditions are appended. This is, e.g., provided by members of the same \textit{associate family} of \textit{minimal surfaces} (i.e., it holds $\EuScript{H}\equiv0$). The most prominent example is the catenoid and helicoid family, see Appendix. More precisely, one can bend the catenoid without stretching into a portion of a helicoid in such a way that the surface normals remain unchanged. For further examples of associate families of minimal surfaces we refer the reader to \cite[Chapter 3]{minimalSurfaces339}. 
\end{remark}

Now we are ready to turn to the large drill rotation case and prove our main result.

\begin{proof}[Proof of Proposition \ref{prop1}]
 Using the boundary condition $m|_{\gamma_d}=y_0|_{\gamma_d}$, Lemma \ref{lem1} allows to conclude that $Q|_{\gamma_d}\equiv \id$ which implies $\sin\alpha|_{\gamma_d}=0$ and $\cos\alpha|_{\gamma_d}=1$. By Lemma \ref{LemAngle2} we have $\sin\alpha\equiv\operatorname{const}$ and $\cos\alpha\equiv\operatorname{const}$, so that $\sin\alpha\equiv0$ and $\cos\alpha\equiv1$ on $\overline{\omega}$, and consequently by the Euler-Rodrigues representation it follows $Q\equiv\id$.
\end{proof}

\begin{proof}[Proof of Corollary \ref{cor:Intro}]
	Consider the lifted quantities $(\D m|n)$ versus $(\D y_0|n_0)$. It holds
	\begin{align}\label{matrix}
		(\D m|n)^T(\D m|n)=
		\left(\begin{NiceArray}{CC|C}
 \Block{2-2}{[\D m]^T\D m} && 0 \\
 &\hspace*{1.5cm}& 0 \\
 \hline
 0 & 0& 1
 \end{NiceArray}\right),\quad
 (\D y_0|n_0)^T(\D y_0|n_0)=
 \left(\begin{NiceArray}{CC|C}
 \Block{2-2}{[\D y_0]^T\D y_0} && 0 \\
 &\hspace*{1.5cm}& 0 \\
 \hline
 0 & 0& 1
 \end{NiceArray}\right)\,.
	\end{align}
	Since by assumption $\mathrm{I}_m=\D m^T\D m=\D y_0^T\D y_0=\mathrm{I}_{y_0}$, it follows from (\ref{matrix}) that 
	\begin{align}\label{matrix2}
	(\D m|n)^T(\D m|n)=(\D y_0|n_0)^T(\D y_0|n_0)\,.
	\end{align}
	Then for all $x\in\omega$ there exists $Q(x)\in \SO(3)$ such that $(\D m(x)|n(x))=Q(x)(\D y_0(x)|n_0(x))$. The assumption $n(x)=n_0(x)$ gives 
	\begin{align}\label{end}
	(\D m|n_0)=Q(\D y_0|n_0)\,.
	\end{align}
	Multiplying both sides with $e_3$, we obtain $Q(x)n_0(x)=n_0(x)$ for all $x\in\omega$ and from \eqref{end} we obtain 
	\begin{align}
	Q=(\D m|n_0)(\D y_0|n_0)^{-1}=(\D m|n_0)\;\frac{1}{\det(\D y_0|n_0)}\,\Cof (\D y_0|n_0)\,.
	\end{align}
	Since by assumption $m,y_0\in C^2(\overline{\omega},\R^3)$ and $y_0$ is a regular surface with $\det(\D y_0|n_0)\geq c^+>0$ we observe that necessarily $Q\in C^1(\overline{\omega},\SO(3))$. 
	Thus, we are again in the situation of Proposition \ref{prop1} and the proof is finished.
\end{proof}

\begin{proposition}\label{prop:meanCurvature2}
 Let $\omega\subset \R^2$ be a bounded Lipschitz domain. Assume that $m,\,y_0\in C^2(\omega,\R^3)$ are regular surfaces, $Q\in C^1(\omega,\SO(3))$ and 
	\begin{align}\label{origin12}
	\D m(x)&=Q(x)\D y_0(x)\,,\quad\quad Q(x)\,n_0(x)=n_0(x)\,,\quad\quad x\in \omega\,,
	\end{align}
	where $n_0=\frac{\partial_1 y_0\times \partial_2 y_0}{\norm{\partial_1 y_0\times \partial_2 y_0}}$ denotes the normal field on $y_0(\omega)$. Then
	\begin{equation}\label{eq:conclH2}
	 \forall \ x \in \omega: \quad \sin\alpha(x)=0 \quad \text{or}\quad \EuScript{H}(x)=0,
	\end{equation}
	where $\alpha:\omega\to\R$ denotes the rotations angle in the Euler-Rodrigues representation of $Q$.
\end{proposition}

\begin{proof}
By Lemma \ref{LemAngle2} we have $\cos \alpha\equiv\operatorname{const}$ and $\sin\alpha\equiv
\operatorname{const}$, so that taking and comparing the mixed derivatives of $m$ we obtain
\begin{equation*}
0 = \sin\alpha\cdot(\underbrace{\partial_1 n_0\times \partial_2 y_0}_{=\vv{A}} - \underbrace{\partial_2 n_0\times \partial_1 y_0}_{=\vv{B}})\overset{\eqref{eq:meanCurvature}}{=} -2\,\sin\alpha\cdot \EuScript{H}\cdot\partial_1 y_0\times \partial_2 y_0\,,
 \end{equation*}
 which implies that we have pointwise either vanishing $\sin \alpha$ or a vanishing mean curvature $\EuScript{H}$. 
\end{proof}

\begin{remark}[Symmetry of the second fundamental form]
 Conclusion \eqref{eq:conclH2} can also be obtained from the symmetry property of the second fundamental form on the surface $m(\omega)$. Indeed, the normal vector field on $m(\omega)$ coincides with $n_0$ since
 \begin{align}
  \partial_1 m \times \partial_2 m \overset{\eqref{origin12}_1}{=} (Q\,\partial_1y_0)\times (Q\,\partial_2y_0)\overset{Q\in\SO(3)}{=}Q\, \partial_1 y_0\times \partial_2 y_0 \overset{\eqref{origin12}_2}{=}\partial_1 y_0\times \partial_2 y_0.
 \end{align}
Thus, for the second fundamental form on $m(\omega)$ we obtain
\begin{align}
 \Sym(2)\ni \mathrm{II}_m&= -[\D m]^T\D n= -[\D m]^T\D n_0 \overset{\eqref{origin12}_1}{=}-[\D y_0]^T Q^T \D n_0 \notag \\
 & \overset{\mathclap{\eqref{Q}}}{=} 
 -[\D y_0]^T [(1-\cos \alpha)\,n_0\otimes n_0+\cos\alpha\,\id+\sin\alpha\,\Anti(n_0)]^T\D n_0 \notag \\
 & = (1-\cos \alpha)\underset{=0}{\underbrace{[\D y_0]^Tn_0\otimes n_0}}\D n_0-\cos\alpha[\D y_0]^T\D n_0-\sin\alpha[\D y_0]^T\Anti(n_0)^T\D n_0\notag \\
 & = \cos\alpha \,\mathrm{II}_{y_0}+ \sin\alpha\,[\D y_0]^T\Anti(n_0)\D n_0.
 \end{align}
 Since $ \mathrm{II}_{y_0}\in\Sym(2)$, we are again left with the single condition
\begin{align}
 \sin\alpha\,[\D y_0]^T \Anti(n_0)\D n_0 \in \Sym(2) \quad 
 \overset{\eqref{eq:ersterSchritt1}}{\Longleftrightarrow} \quad \{\sin\alpha = 0 \quad \text{or} \quad \EuScript{H}=0\}.
\end{align}
\end{remark}

\begin{corollary}
 Let $\omega\subset \R^2$ be a bounded Lipschitz domain. Assume that $m,\,y_0\in C^2(\omega,\R^3)$ are regular surfaces, $Q\in C^1(\omega,\SO(3))$ and 
	\begin{align}
	\D m(x)&=Q(x)\D y_0(x)\,,\quad\quad Q(x)\,n_0(x)=n_0(x)\,,\quad\quad x\in \omega\,,
	\end{align}
	where $n_0=\frac{\partial_1 y_0\times \partial_2 y_0}{\norm{\partial_1 y_0\times \partial_2 y_0}}$ denotes the normal field on $y_0(\omega)$. If $y_0$ is not a minimal surface, i.e., there exists one point $x_0$ such that the mean curvature of $y_0$ is distinct from zero at $y_0(x_0)$, then $m(x)=y_0(x)+b$ or   $m(x)=-y_0(x)+b$  for some constant translation $b\in\R^3$.
\end{corollary}

\begin{proof}
 It follows from Propostion \ref{prop:meanCurvature2}, that if the mean curvature $\EuScript{H}$ is distinct from $0$ at some point $y_0(x_0)$, we must have $\sin\alpha(x_0)= 0$. Thus, $\sin\alpha\equiv0$ and $\cos\alpha\equiv1$ or  $\sin\alpha\equiv0$ and $\cos\alpha\equiv-1$, where the sign of $\cos\alpha$ does not change, cf.~ Lemma \ref{LemAngle2}. Thus, in the first case ($\cos\alpha\equiv1$, i.e., $\alpha\in2\pi\mathbb Z$) we obtain for
 \begin{equation}\label{eq:jdhlkjehdl}
  Q(x)\overset{\eqref{Q}}{=}(1-\cos\alpha)n_0(x)\otimes n_0(x)+ \id\,,
 \end{equation}
and a multiplication with $\D y_0$ gives
\begin{equation*}
 \D m(x)= Q(x)\D y_0(x)\overset{\eqref{eq:jdhlkjehdl}}{=} (1-\cos\alpha)\underset{=0}{\underbrace{n_0(x)\otimes n_0(x)\D y_0(x)}} +\D y_0(x) =\D y_0(x) .
 \end{equation*}
 In the second case ($\cos\alpha\equiv-1$, i.e., $\alpha-\pi\in2\pi\mathbb Z$) we have
 \begin{equation*}
  Q(x)\overset{\eqref{Q}}{=}(1-\cos\alpha)n_0(x)\otimes n_0(x)- \id\quad \Rightarrow \quad \D m(x)= -\D y_0(x)\,.\qedhere
 \end{equation*}
\end{proof}

\begin{remark}
 The ``flipped'' solution $m(x)=-y_0(x)+b$ appears only since no boundary conditions are present. From a mechanical point of view this branch is irrelevant.
\end{remark}

\begin{remark}
 	For a $C^\infty$-embedding $y_0$ a comparable result, using involved techniques from differential geometry, has been obtained in \cite{abe1975isometric, eschenburg2010compatibility}.
\end{remark}

\section{Conclusion}\label{conclusion}
We have proved some improved rigidity results for $C^2$-smooth regular embedded surfaces. The underlying mechanical problem, namely the possibility of a pure in-plane drill rotation field as deformation mode of a surface when boundary conditions of place are prescribed somewhere has been negatively answered, for both small and large drill rotations, assuming some natural level of smoothness for the rotation fields. 

Considering classical FEM shell formulations the drilling degrees of freedom are used to obtain a precise coupling of plane shell elements in non-plane applications, in such a way, that rotations around an axis in the plane of one element are coupled to rotations around the normal of the neighboring element. Since finite elements use their shape functions as interpolations, they do not obey the kinematics everywhere. Thus, incompatibilities in the shape functions may occur, but, should not get in conflict
with the convergence requirements such as the patch test and any "torsional spring stiffness" attributed to this pseudo-deformation mode must be regarded with extreme caution (here, the Cosserat couple modulus $\mu_c\geq 0$). In fact, in many shell models such a stiffness is treated as a material parameter and many very effective finite elements use incompatible kinematic fields, mainly to overcome geometric deficiencies which may result in some kind of locking behavior. Our development suggests, however, that the fitting of the Cosserat couple modulus $\mu_c$ would depend on the imposed boundary condition, i.e., how strict drill rotations are constrained at the boundary $\gamma_{d}$ of the shell.

Then, the mentioned stiffness is a boundary value problem dependent parameter which needs to be determined for each new problem again. Thus one should call it always a fictitious stifness and treat it accordingly, and this applies to classical FEM-shell models and Cosserat surfaces.

 However, it is remarkable that in the planar Cosserat shell model \cite{agn_neff2004geometrically} existence can be shown also for zero Cosserat couple modulus $\mu_c\equiv 0$ \cite{agn_neff2007geometricallyI} (for $q>0$ in \eqref{mimized} and using a generalized Korn-inequality \cite{agn_neff2002korn}) thus disposing completely of the above described problem. In other words, the drilling degree of freedom is kept, but not connected  to any fictitious torsional spring. While in a linear model this would imply that the drilling degree of freedom decouples, this is not necessarily the case in nonlinear Cosserat models based on exact rotations.
 
\begin{akn*}
	The last author is indebted to Amit Acharya (Carnegie Mellon University) for clarifying discussions regarding his notion of normality preserving shell deformations. We also thank Robert L. Bryant (Duke University, North-Carolina) for fruitful discussions. Discussions with Marek L. Szwabowicz (Maritime University, Gdansk, Poland) and Krzysztof  Wiśniewski (Institute of Fundamental Technological Research               Polish Academy of Sciences, Warsaw, Poland)  are also acknowledged. This research has been supported by the Deutsche Forschungsgemeinschaft (DFG, German Research Foundation)-Project No. \textbf{415894848}, \;\textbf{Neff 902/8-1}\;(P. Lewintan and P. Neff).	Maryam Mohammadi Saem is grateful for a grant of the Faculty of Mathematics, University Duisburg-Essen. \\
\end{akn*} 

\section{References}
\printbibliography[heading=none]

 \begin{alphasection}
\section{Appendix}
\subsection{Associate family of minimal surfaces: the catenoid and helicoid family}\label{family}
Meusnier in 1776 discovered that the helicoid and the catenoid are minimal surfaces, i.e., they have vanishing mean curvature. It was Bonnet who showed that these minimal surfaces belong to the same associate family, more precisely, the catenoid can be bend without stretching into a portion of a helicoid in such a way that the surface normals remain unchanged. For further examples of associate families of minimal surfaces we refer the reader to \cite[Chapter 3]{minimalSurfaces339}. 

Let $\omega=\R\times[0,2\pi)$ then  the catenoid $X^{\text{cat}}\in C^\infty(\omega,\R^3)$ and the helicoid $X^\text{hel}\in C^\infty(\omega,\R^3)$ are parametrized by
\begin{align}
X^\text{cat}(x_1,x_2)=\matr{\cos x_2\cosh x_1\\\sin x_2\cosh x_1\\x_1}\,,\qquad \quad X^\text{hel}(x_1,x_2)=\matr{\hspace{0.2cm}\sin x_2\sinh x_1\\-\cos x_2\sinh x_1\\x_2}\,,
\end{align}
and the associate surfaces, parametrized by $\theta$, are given by
\begin{align}\label{eq:associate}
X^\theta(x_1,x_2)=\cos\theta\cdot X^\text{cat}(x_1,x_2)+\sin\theta\cdot X^\text{hel}(x_1,x_2)\, \quad \text{for $\theta\in[0,2\pi]$}.
\end{align}
Their partial derivatives fulfill
\begin{align}\label{eq:partials}
\partial_1 X^\theta=\cos\theta\cdot\partial_1 \, X^\text{cat}-\sin\theta\cdot\partial_2\, X^\text{cat}\quad\text{and}\quad
\partial_2 X^\theta=\sin\theta\cdot\partial_1 \, X^\text{cat}+\cos\theta\cdot\partial_2\, X^\text{cat}\,,
\end{align}
 so that, the surface normals remain unchanged
\begin{equation}\label{eq:normals}
 n_{X^\theta}(x_1,x_2) = n_{X^\text{cat}}(x_1,x_2)=n_{X^\text{hel}}(x_1,x_2) \qquad \text{for all $(x_1,x_2)\in\omega$.}
\end{equation}
Further properties of the members of this associate family $X^\theta$ can be found in \cite{karcher1996construction} and \cite[Chapter 3]{minimalSurfaces339}, as well as the references cited therein.

\begin{SCfigure}[][h!]
\begin{subfigure}[b]{2.7cm}
 \includegraphics[page=1,width=2.5cm]{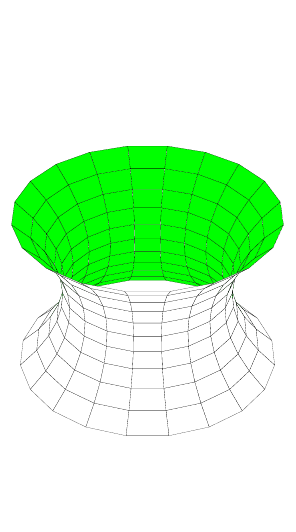}
 \caption{$\theta=0$}
\end{subfigure}
\begin{subfigure}[b]{2.7cm}
 \includegraphics[page=2,width=2.5cm]{associate_family}
 \caption{$\theta=\frac{\pi}{25}$}
\end{subfigure}
\begin{subfigure}[b]{2.7cm}
 \includegraphics[page=3,width=2.5cm]{associate_family}
 \caption{$\theta=\frac\pi4$}
\end{subfigure}
\begin{subfigure}[b]{2.7cm}
 \includegraphics[page=4,width=2.5cm]{associate_family}
 \caption{$\theta=\frac\pi2$}
\end{subfigure}
 \caption{Four consecutive steps of an isometric deformation of a catenoid (left) into a helicoid (right). Such a transformation exists, since both are members of the same associate family $X^\theta$. Note that every member of the deformation family has vanishing mean curvature, i.e., is a minimal surface.}
\end{SCfigure}

To see \eqref{eq:partials} and \eqref{eq:normals} we take the partial derivatives of \eqref{eq:associate}:
\begin{equation}
 \partial_j X^\theta = \cos\theta\cdot \partial_j X^\text{cat}+\sin\theta\cdot \partial_j X^\text{hel} \qquad \text{for $j=1,2$.}
\end{equation}
Since the partial derivatives of the catenoid and the helicoid satisfy the Cauchy-Riemann equations
\begin{equation}
 \partial_1 X^\text{cat}= \partial_2 X^\text{hel} \quad \text{and}\quad  \partial_2 X^\text{cat}=- \partial_1 X^\text{hel}\,,
\end{equation}
we obtain \eqref{eq:partials}, which in matrix notation reads
\begin{align}\label{eq:matrixnotation3123}
 \D X^\theta = \D X^\text{cat}\matr{\cos\theta & \sin\theta \\ -\sin\theta & \cos\theta}.
\end{align}
Moreover,
\begin{align}
 \partial_1 X^\theta \times \partial_2 X^\theta \ &\overset{\mathclap{\eqref{eq:partials}}}{=}\ \cos^2\theta\cdot \partial_1 X^\text{cat} \times \partial_2 X^\text{cat} - \sin^2\theta \cdot \partial_2 X^\text{cat} \times \partial_1 X^\text{cat} \notag\\
 &= \partial_1 X^\text{cat} \times \partial_2 X^\text{cat}\,,
\end{align}
which gives \eqref{eq:normals}. Furthermore,
\begin{equation}\label{eq:concretepartials}
 \partial_1 X^{\text{cat}}=\matr{\cos x_2\sinh x_1\\\sin x_2\sinh x_1\\1} \quad\text{and}\quad \partial_2 X^{\text{cat}}=\matr{-\sin x_2\cosh x_1\\\cos x_2\cosh x_1\\0},
\end{equation}
so that
\begin{equation}\label{eq:partial_cat}
 \norm{ \partial_1 X^{\text{cat}}}^2= \sinh^2x_1 +1 = \cosh^2x_1 = \norm{ \partial_2 X^{\text{cat}}}^2 \quad\text{and}\quad \skalarProd{\partial_1 X^{\text{cat}}}{\partial_2 X^{\text{cat}}}=0.
\end{equation}
In regard with \eqref{eq:partials} we obtain
\begin{subequations}
\begin{align}
 \norm{ \partial_1 X^\theta}^2&\overset{\eqref{eq:partial_cat}_2}{=}\cos^2\theta \norm{ \partial_1 X^{\text{cat}}}^2 + \sin^2\theta\norm{ \partial_2 X^{\text{cat}}}^2 \overset{\eqref{eq:partial_cat}_1}{=} \cosh^2 x_1\,,\\
  \norm{ \partial_2 X^\theta}^2&\overset{\eqref{eq:partial_cat}_2}{=}\sin^2\theta \norm{ \partial_1 X^{\text{cat}}}^2 + \cos^2\theta\norm{ \partial_2 X^{\text{cat}}}^2 \overset{\eqref{eq:partial_cat}_1}{=} \cosh^2 x_1=\norm{ \partial_1 X^\theta}^2\,,\\
  \skalarProd{\partial_1 X^\theta}{\partial_2 X^\theta}&\overset{\eqref{eq:partial_cat}_2}{=}\cos\theta\sin\theta\norm{ \partial_1 X^{\text{cat}}}^2 -\sin\theta\cos\theta\norm{ \partial_2 X^{\text{cat}}}^2\overset{\eqref{eq:partial_cat}_1}{=} 0.
\end{align}
\end{subequations}
In other words, the first fundamental form of all members of the associate family remains unchanged and is given by
\begin{equation}\label{eq:1ff_associate}
\mathrm{I}_{X^\theta}(x_1,x_2)= [\D X^\theta]^T\D X^\theta = \matr{\norm{ \partial_1 X^\theta}^2 &  \skalarProd{\partial_1 X^\theta}{\partial_2 X^\theta} \\ \skalarProd{\partial_1 X^\theta}{\partial_2 X^\theta} &  \norm{ \partial_2 X^\theta}^2} = \cosh^2 x_1\cdot \id_2.
\end{equation}
Thus,
\begin{equation}
 [\D X^\text{cat}]^T\D X^\theta \overset{\eqref{eq:matrixnotation3123}}{=}[\D X^\text{cat}]^T\D X^\text{cat}\matr{\cos\theta & \sin\theta \\ -\sin\theta & \cos\theta} \overset{\eqref{eq:1ff_associate}}{=}\cosh^2 x_1\matr{\cos\theta & \sin\theta \\ -\sin\theta & \cos\theta}\,,
\end{equation}
and $[\D X^\text{cat}]^T\D X^\theta\notin\Sym^+(2)$ is not a pure in-plane stretch.

Moreover, as in \eqref{matrix2}, there exists an in-plane drill rotation $Q^\theta(x)\in\SO(3)$ which fulfills
\begin{equation}
 \D X^\theta(x) = Q^\theta(x)\D X^\text{cat}(x) \quad \text{and}\quad Q^\theta(x)\,n_0(x)=n_0(x), \quad \text{where } n_0(x)\coloneqq n_{X^\theta}(x)=n_{X^\text{cat}}(x).
\end{equation}
Next we show, that the (constant) rotation angle, cf.~Lemma \ref{LemAngle2}, extracted from $Q^\theta(x)$ is already given by $-\theta$ (or differs from it by an integer multiple of $2\pi$), so that we have the representation
\begin{align}
 Q^\theta(x)&=(1-\cos(-\theta))\,n_0(x)\otimes n_0(x)+\cos(-\theta)\,\id+\sin(-\theta)\Anti(n_0(x))\notag \\
 & = (1-\cos\theta)\,n_0(x)\otimes n_0(x)+\cos\theta\,\id-\sin\theta\Anti(n_0(x)).\label{eq:Qtheta}
\end{align}
For that purpose, note that
\begin{equation}
 \partial_1 X^\text{cat} \times \partial_2 X^\text{cat} = \matr{-\cos x_2\cosh x_1 \\ -\sin x_2 \cosh x_1 \\ \sinh x_1 \cosh x_1 } \quad \Rightarrow \quad \norm{\partial_1 X^\text{cat} \times \partial_2 X^\text{cat}} = \cosh^2 x_1,
\end{equation}
so that
\begin{equation}\label{eq:mitnomalen}
 n_0 = \frac{1}{\cosh x_1}\matr{-\cos x_2\\ -\sin x_2 \\ \sinh x_1}, \quad \text{and}\quad n_0\times\partial_1 X^{\text{cat}}= \partial_2 X^{\text{cat}}, \quad n_0\times\partial_2 X^{\text{cat}}= -\partial_1 X^{\text{cat}}.
\end{equation}
Hence, with $n_0\otimes n_0\, \partial_j X^{\text{cat}}=n_0\,\skalarProd{n_0}{\partial_j X^{\text{cat}}}\equiv0$, we obtain
\begin{align}
 Q^\theta \partial_1 X^{\text{cat}} &\overset{\eqref{eq:Qtheta}}{=}\cos\theta\,\partial_1 X^{\text{cat}}-\sin\theta\Anti(n_0)\partial_1 X^{\text{cat}}\overset{\eqref{eq:mitnomalen}_2}{=}\cos\theta\,\partial_1 X^{\text{cat}}-\sin\theta\, \partial_2 X^{\text{cat}} \overset{\eqref{eq:partials}_1}{=}\partial_1 X^\theta\,,
 \intertext{as well as}
 Q^\theta \partial_2 X^{\text{cat}} &\overset{\eqref{eq:Qtheta}}{=}\cos\theta\,\partial_2 X^{\text{cat}}-\sin\theta\Anti(n_0)\partial_2 X^{\text{cat}}\overset{\eqref{eq:mitnomalen}_3}{=}\cos\theta\,\partial_2 X^{\text{cat}}+\sin\theta\, \partial_1 X^{\text{cat}} \overset{\eqref{eq:partials}_2}{=}\partial_2 X^\theta\,,
 \end{align}
and we have shown that 
\begin{equation}
 Q^\theta \D X^{\text{cat}}=\D X^\theta\,,
\end{equation}
where $Q^\theta$ has the expression \eqref{eq:Qtheta} and its columns read with the representation of the normal \eqref{eq:mitnomalen}$_1$
\begin{subequations}\label{eq:Qincoords}
\begin{align}
 Q^\theta\,e_1&=
 \begin{pmatrix}-\frac{\left( {\cos^2 x_2}-{\cosh^2 x_1}\right)  \cos{\theta }-{\cos^2 x_2}}{{\cosh^2 x_1}} \\[6pt]
-\frac{\cosh{x_1} \sinh{x_1} \sin{\theta }+\cos{x_2} \sin{x_2} \cos{\theta }-\cos{x_2} \sin{x_2}}{{\cosh^2 x_1}} \\[6pt]
-\frac{\cosh{x_1} \sin{x_2} \sin{\theta }-\sinh{x_1} \cos{x_2} \cos{\theta }+\sinh{x_1} \cos{x_2}}{{\cosh^2 x_1}}\end{pmatrix}\,,\\
 Q^\theta\,e_2&=
 \begin{pmatrix} \frac{\cosh{x_1} \sinh{x_1} \sin{\theta }-\cos{x_2} \sin{x_2} \cos{\theta }+\cos{x_2} \sin{x_2}}{{\cosh^2 x_1}} \\[6pt]
-\frac{\left( {\sin^2 x_2}-{\cosh^2 x_1}\right)  \cos{\theta }-{\sin^2 x_2}}{{\cosh^2 x_1}} \\[6pt]
 \frac{\cosh{x_1} \cos{x_2} \sin{\theta }+\sinh{x_1} \sin{x_2} \cos{\theta }-\sinh{x_1} \sin{x_2}}{{\cosh^2 x_1}} \end{pmatrix}\,,\\
Q^\theta\,e_3&=
 \begin{pmatrix} \frac{\cosh{x_1} \sin{x_2} \sin{\theta }+\sinh{x_1} \cos{x_2} \cos{\theta }-\sinh{x_1} \cos{x_2}}{{\cosh^2 x_1}}\\[6pt]
-\frac{\cosh{x_1} \cos{x_2} \sin{\theta }-\sinh{x_1} \sin{x_2} \cos{\theta }+\sinh{x_1} \sin{x_2}}{{\cosh^2 x_1}}\\[6pt]
 \frac{\cos{\theta }+{\cosh^2 x_1}-1}{{\cosh^2 x_1}}\end{pmatrix}\,.
 \end{align}
 \end{subequations}
Recall, that if a surface $X$ is parametrized conformally, i.e., it holds $\norm{\partial_1 X}=\norm{\partial_2 X}$ and $\skalarProd{\partial_1 X}{\partial_2 X}=0$, then it is a minimal surface (i.e.~has vanishing mean curvature everywhere) if and only if ~ $\Delta X \equiv 0$ ~ holds, cf.~\cite[p.72]{minimalSurfaces339}. Thus, in regard with \eqref{eq:1ff_associate}, to check that all members of the associate family $X^\theta$ are, indeed, minimal surfaces, we compute
\begin{equation}
 \Delta X^\theta \overset{\eqref{eq:partials}}{=} \cos\theta\cdot \Delta X^\text{cat} \overset{\eqref{eq:concretepartials}}{\equiv}0.
\end{equation}
Furthermore, if a minimal surface is parametrized conformally, then the same holds for its corresponding Gauss map, cf.~\cite[p.74]{minimalSurfaces339}. Indeed, all members of the associate family $X^\theta$ are minimal surfaces and for their Gauss maps (which all coincide) $n_0:\omega\to\mathbb{S}^2$ it holds
\begin{equation}
 \mathrm{I}_{n_0}=[\D n_0]^T\D n_0\overset{\eqref{eq:mitnomalen}_1}{=}\frac{1}{\cosh^2x_1}\cdot\id_2\,,
 \end{equation}
 which shows that $n_0$ is also parametrized conformally.
 
 Let us mention, that the constancy of the rotation angle can also be achieved without applying Lemma \ref{LemAngle2}. For that purpose, let us here call the in-plane drill rotation $\widehat{Q}(x)\in\SO(3)$ which fulfills
\begin{equation}
 \D X^\theta(x) = \widehat{Q}(x)\D X^\text{cat}(x) \quad \text{and}\quad \widehat{Q}(x)\,n_0(x)=n_0(x), \quad \text{where } n_0(x)\coloneqq n_{X^\theta}(x)=n_{X^\text{cat}}(x).
\end{equation}
Thus, as in \eqref{matrix2}, it follows
\begin{equation}
 (\D X^\theta \,\big|\,  n_0) = \widehat{Q} (\D X^\text{cat}\,\big|\, n_0) \quad \Rightarrow \quad \widehat{Q} = (\D X^\theta \,\big|\, n_0) (\D X^\text{cat}\,\big|\, n_0)^{-1}\,,
\end{equation}
and a direct computation gives the entries of $\widehat{Q}(x)$, which, indeed, coincide with \eqref{eq:Qincoords}. From the uniqueness of the Euler-Rodrigues representation it then follows that the corresponding rotation angle $\widehat{\alpha}(x)$ is constant and is given by $-\theta$ (or differs from it by an integer multiple of $2\pi$). Indeed, with \eqref{eq:tan_expression} we have
\begin{equation*}
 \sin(\widehat{\alpha}(x))\overset{\eqref{eq:tan_expression}}{=}-\frac{\tr\Big(\Anti(n_0(x))\widehat{Q}(x)\Big)}{2}\underset{\eqref{eq:Qincoords}}{\overset{\widehat{Q}=Q^\theta}{=}}-\sin\theta \quad \text{and} \quad 
  \cos(\widehat{\alpha}(x))\overset{\eqref{eq:tan_expression}}{=}-\frac{\tr\widehat{Q}(x)-1}{2}\underset{\eqref{eq:Qincoords}}{\overset{\widehat{Q}=Q^\theta}{=}}\cos\theta.
\end{equation*}

\end{alphasection}

\end{document}